\documentclass{amsart}
\usepackage{amssymb}
\usepackage{mathrsfs}
\usepackage{stmaryrd}
\usepackage{imakeidx}
\usepackage{hyperref}
\usepackage{enumitem}
\input xy \xyoption {all}

\usepackage{tikz}
\usetikzlibrary{cd}

\usepackage{dsfont}

\makeindex

\newcommand{\Z}{\mathbb{Z}}

\newcommand{\F}{\mathbb{F}}

\newcommand{\K}{\mathbb{K}}
\renewcommand{\L}{\mathbb{L}}
\renewcommand{\O}{\mathbb{O}}
\newcommand{\Ql}{\mathbb{Q}_\ell}

\newcommand{\bk}{\Bbbk}
\newcommand{\Ga}{\mathbb{G}_{\mathrm{a}}}
\newcommand{\Gm}{\mathbb{G}_{\mathrm{m}}}

\newcommand{\Db}{D^{\mathrm{b}}}

\newcommand{\bX}{\mathbf{X}}

\newcommand{\cE}{\mathcal{E}}

\newcommand{\Perv}{\mathrm{Perv}}

\newcommand{\Tilt}{\mathrm{Tilt}}

\newcommand{\cP}{\mathcal{P}}

\newcommand{\cF}{\mathcal{F}}
\newcommand{\cG}{\mathcal{G}}

\newcommand{\cL}{\mathcal{L}}

\newcommand{\Av}{\mathsf{Av}}
\newcommand{\For}{\mathsf{For}}
\newcommand{\IC}{\mathrm{IC}}

\newcommand{\id}{\mathrm{id}}

\newcommand{\simto}{\overset{\sim}{\to}}

\DeclareMathOperator{\Hom}{Hom}
\DeclareMathOperator{\Ext}{Ext}

\newcommand{\Fl}{\mathrm{Fl}}
\newcommand{\Gr}{\mathrm{Gr}}
\newcommand{\Wf}{W_{\mathrm{f}}}

\newcommand{\IW}{\mathcal{IW}}

\makeatletter
\def\lotimes{\@ifnextchar_{\@lotimessub}{\@lotimesnosub}}
\def\@lotimessub_#1{\mathchoice{\mathbin{\mathop{\otimes}^L}_{#1}}%
  {\otimes^L_{#1}}{\otimes^L_{#1}}{\otimes^L_{#1}}}
\def\@lotimesnosub{\mathbin{\mathop{\otimes}^L}}
\makeatother

\makeatletter
\def\lboxtimes{\@ifnextchar_{\@lboxtimessub}{\@lboxtimesnosub}}
\def\@lboxtimessub_#1{\mathchoice{\mathbin{\mathop{\boxtimes}^L}_{#1}}%
  {\boxtimes^L_{#1}}{\boxtimes^L_{#1}}{\boxtimes^L_{#1}}}
\def\@lboxtimesnosub{\mathbin{\mathop{\boxtimes}^L}}
\makeatother

\newcommand{\scO}{\mathscr{O}}
\newcommand{\scK}{\mathscr{K}}
\newcommand{\cJ}{\mathcal{J}}
\newcommand{\pH}{{}^p \hspace{-1pt} \mathcal{H}}
\newcommand{\Rep}{\mathrm{Rep}}
\newcommand{\GO}{G_\scO}
\newcommand{\GK}{G_\scK}

\newcommand{\rInd}{{}^* \hspace{-0.5pt} \mathsf{Ind}}
\newcommand{\lInd}{{}^! \hspace{-0.5pt} \mathsf{Ind}}

\newcommand{\uu}{\mathrm{u}}

\numberwithin{equation}{section}
\numberwithin{figure}{section}
\newtheorem{thm}{Theorem}[section]
\newtheorem{lem}[thm]{Lemma}
\newtheorem{prop}[thm]{Proposition}
\newtheorem{cor}[thm]{Corollary}

\theoremstyle{definition}

\theoremstyle{remark}
\newtheorem{rmk}[thm]{Remark}

\title{An Iwahori--Whittaker model for the Satake category}

\author[R. Bezrukavnikov]{Roman Bezrukavnikov}
\address{Department of Mathematics \\ Massachusetts Institute of Technology \\ Cambridge, MA \\ 02139 \\ USA.}
\email{bezrukav@math.mit.edu}

\author[D. Gaitsgory]{Dennis Gaitsgory}
\address{Harvard University \\ 1 Oxford St \\ Cambridge, MA \\ 02138 \\ USA.}
\email{gaitsgde@math.harvard.edu}

\author[I.~Mirkovi\'c]{Ivan Mirkovi\'c}
\address{University of Massachusetts, Amherst, MA, USA.}
\email{mirkovic@math.umass.edu}

 \author[S.~Riche]{Simon Riche}
 \address{Universit\'e Clermont Auvergne, CNRS, LMBP, F-63000 Clermont-Ferrand, France.}
 \email{simon.riche@uca.fr}
 
 \author[L.~Rider]{Laura Rider}
 \address{Department of Mathematics, University of Georgia, Athens Georgia 30602, USA.}
 \email{laurajoy@uga.edu}
 
 \thanks{R.B. was partially supported by NSF grant No. DMS-1601953. D.G. was supported by NSF Grant No. DMS-1063470. This project has received funding from the European Research Council (ERC) under the European Union's Horizon 2020 research and innovation programme (S.R., grant agreement No. 677147). L.R. was supported by NSF Grant No. DMS-1802378.}

\begin{document}

\begin{abstract}
 In this paper we prove, for $G$ a connected reductive algebraic group satisfying a mild technical assumption, that the Satake category of $G$ (with coefficients in a finite field, a finite extension of $\Ql$, or the ring of integers of such a field) can be described via Iwahori--Whittaker perverse sheaves on the affine Grassmannian. As applications, we confirm a conjecture of Juteau--Mautner--Williamson describing the tilting objects in the Satake category, and give a new proof of the property that a tensor product of tilting modules is tilting.
\end{abstract}

\maketitle

\section{Introduction}

\subsection{Another incarnation of the Satake category}

Let $G$ be a connected reductive algebraic group over an algebraically closed field $\F$ of positive characteristic, and let $\bk$ be either a finite field of characteristic $\ell \neq \mathrm{char}(\F)$, or a finite extension of $\Ql$, or the ring of integers of such an extension. If $\scK:=\F( \hspace{-0.5pt} (z) \hspace{-0.5pt} )$ and $\scO:= \F[ \hspace{-0.5pt} [z] \hspace{-0.5pt} ]$, the \emph{Satake category} is the category
\[
 \Perv_{\GO}(\Gr,\bk)
\]
of $\GO$-equivariant (\'etale) $\bk$-perverse sheaves on the affine Grassmannian
\[
 \Gr:=\GK/\GO
\]
of $G$. This category is a fundamental object in Geometric Representation Theory through its appearance in the \emph{geometric Satake equivalence}, which claims that this category admits a natural convolution product $(-) \star^{\GO} (-)$, which endows it with a monoidal structure, and that there exists an equivalence of monoidal categories
\begin{equation}
\label{eqn:Satake}
 \mathcal{S} : (\Perv_{\GO}(\Gr,\bk), \star^{\GO}) \xrightarrow{\sim} (\Rep(G^\vee_\bk), \otimes).
\end{equation}
 Here the right-hand side is the category of algebraic representations of the split reductive $\bk$-group scheme which is Langlands dual to $G$ on finitely generated $\bk$-modules; see~\cite{mv} for the original proof of this equivalence in full generality, and~\cite{br} for a more detailed exposition. (In these references, what is explicitly treated is the analogous equivalence for a \emph{complex} group $G$, in which case $\bk$ can be any Noetherian commutative ring of finite global dimension. The \'etale setting is similar; see~\cite[\S 14]{mv} and~\cite[\S 1.1.4]{br} for a few comments.)

This category already has another incarnation since (as proved by Mirkovi\'c--Vilonen) the forgetful functor
\[
 \Perv_{\GO}(\Gr,\bk) \to \Perv_{(\GO)}(\Gr,\bk)
\]
from the Satake category to the category of perverse sheaves on $\Gr$ which are constructible with respect to the stratification by $\GO$-orbits is an equivalence of categories. 

The first main result of the present paper provides a third incarnation of this category, as a category $\Perv_{\IW}(\Gr,\bk)$ of Iwahori--Whittaker\footnote{This terminology is taken from~\cite{ab}. In~\cite{abbgm}, the term ``baby Whittaker'' is used for the same construction.} perverse sheaves on $\Gr$. More precisely we prove that a natural functor
\begin{equation}
\label{eqn:functor-equiv-intro}
 \Perv_{\GO}(\Gr,\bk) \to \Perv_{\IW}(\Gr,\bk)
\end{equation}
is an equivalence of categories, see Theorem~\ref{thm:equiv}. This result is useful because computations in $\Perv_{\IW}(\Gr,\bk)$ are much easier than in the categories $\Perv_{\GO}(\Gr,\bk)$ or $\Perv_{(\GO)}(\Gr,\bk)$, in particular due to the facts that standard/costandard objects have more explicit descriptions and that the ``realization functor''
\[
 \Db \Perv_{\IW}(\Gr,\bk) \to \Db_{\IW}(\Gr,\bk)
\]
is an equivalence of triangulated categories.

In the analogous setting of Whittaker $\mathcal{D}$-modules over a field of characteristic $0$, this statement already 
appears in~\cite{abbgm}.
See Remark~\ref{rmk:definition-categories}\eqref{it:Whittaker-analytic} below for a discussion of possible variants for constructible
sheaves over $\mathbb{C}$. Let us also mention the conjecture~\cite[Conjecture~59]{be} containing 
this statement as a special case (see~\cite[Example~60]{be} for more details).


\subsection{Relation with the Finkelberg--Mirkovi\'c conjecture}
\label{ss:intro-fm}

One possible justification for the equivalence~\eqref{eqn:functor-equiv-intro} comes from a singular analogue of the Finkelberg--Mirkovi\'c conjecture~\cite{fm}. This conjecture states that, if $\bk$ is a field of positive characteristic $\ell$, if $I \subset \GO$ is an Iwahori subgroup and $I_\uu \subset I$ is its pro-unipotent radical, there should exist an equivalence of abelian categories
\[
 \mathsf{F}: \Perv_{I_\uu}(\Gr,\bk) \xrightarrow{\sim} \Rep_0(G^\vee_\bk)
\]
between the category of $I_\uu$-equivariant $\bk$-perverse sheaves on $\Gr$ and the ``extended principal block'' $\Rep_0(G^\vee_\bk)$ of $\Rep(G^\vee_\bk)$, i.e.~the subcategory consisting of modules over which the Harish--Chandra center of the enveloping algebra of the Lie algebra of $G^\vee_\bk$ acts with generalized character $0$.
This equivalence is expected to be compatible with the geometric Satake equivalence in the sense that for $\cF$ in $\Perv_{I_\uu}(\Gr,\bk)$ and $\cG$ in $\Perv_{\GO}(\Gr,\bk)$ we expect a canonical isomorphism
\[
 \mathsf{F}(\cF \star^{\GO} \cG) \cong \mathsf{F}(\cF) \otimes \mathcal{S}(\cG)^{(1)}.
\]
(Here $(-) \star^{\GO} (-)$ is the natural convolution action of $\Perv_{\GO}(\Gr,\bk)$ on the category $\Perv_{I_\uu}(\Gr,\bk)$, and $(-)^{(1)}$ is the Frobenius twist.)

One might expect similar descriptions for some singular ``extended blocks'' of $\Rep(G^\vee_\bk)$, namely those attached to weights in the closure of the fundamental alcove belonging only to walls parametrized by (non-affine) simple roots, involving some Whittaker-type perverse sheaves.\footnote{This extension of the Finkelberg--Mirkovi\'c conjecture stems from discussions of the fourth author with P. Achar. ``Graded versions'' of such equivalences are established in~\cite{acr}.} In the ``most singular'' case, this conjecture postulates the existence of an equivalence
\[
 \mathsf{F}_{\mathrm{sing}} : \Perv_{\IW}(\Gr,\bk) \xrightarrow{\sim} \Rep_{-\varsigma}(G^\vee_\bk)
\]
between our category of Iwahori--Whittaker perverse sheaves and the extended block of weight $-\varsigma$, where $\varsigma$ is a weight whose pairing with any simple coroot is $1$ (the ``Steinberg block''), which should satisfy
\[
 \mathsf{F}_{\mathrm{sing}}(\cF \star^{\GO} \cG) \cong \mathsf{F}_{\mathrm{sing}}(\cF) \otimes \mathcal{S}(\cG)^{(1)}.
\]
(Here we assume that $\varsigma$ exists, which holds e.g.~if the derived subgroup of $G^{\vee}_\bk$ is simply-connected.)

On the representation-theoretic side, it is well known
that the assignment $V \mapsto \mathsf{L}((\ell-1)\varsigma) \otimes V^{(1)}$ induces an equivalence of categories
\[
 \Rep(G^\vee_\bk) \xrightarrow{\sim} \Rep_{-\varsigma}(G^\vee_\bk),
\]
where $\mathsf{L}((\ell-1)\varsigma)$ is the simple $G^\vee_\bk$-module of highest weight $(\ell-1)\varsigma$; see~\cite[\S II.10.5]{jantzen} or~\cite{andersen}.
Our equivalence~\eqref{eqn:functor-equiv-intro} can be considered a geometric counterpart of this equivalence.

\subsection{Relation with results of Lusztig}

Another hint for the equivalence~\eqref{eqn:functor-equiv-intro} is given by some results of Lusztig~\cite{lusztig}. Namely, in~\cite[\S 6]{lusztig} Lusztig defines some submodules $\mathcal{K}$ and $\mathcal{J}$ of (a localization of) the affine Hecke algebra $\mathcal{H}$ attached to $G$. By construction $\mathcal{K}$ is a (non unital) subalgebra of the localization of $\mathcal{H}$, and $\mathcal{J}$ is stable under right multiplication by $\mathcal{K}$. Then~\cite[Corollary~6.8]{lusztig} states 
that $\mathcal{J}$ is free as a right based $\mathcal{K}$-module (for some natural bases), with a canonical generator denoted $J_\rho$. Now $\mathcal{H}$ (or rather its specialization at $q=1$) is categorified by the category of Iwahori-equivariant perverse sheaves on the affine flag variety $\Fl$ of $G$. The subalgebra $\mathcal{K}$ (or rather again its specialization) is then categorified by $\Perv_{\GO}(\Gr,\bk)$ (via the pullback functor to $\Fl$), and similarly $\mathcal{J}$ is categorified by $\Perv_{\IW}(\Gr, \bk)$. From this perspective, the functor in~\eqref{eqn:functor-equiv-intro} is a categorical incarnation of the map $k \mapsto J_\rho \cdot k$ considered by Lusztig, and the fact that it is an equivalence can be seen as a categorical upgrade of~\cite[Corollary~6.8]{lusztig}.

\subsection{Relation with results of Frenkel--Gaitsgory--Kazh\-dan--Vilonen}

Finally, a third hint for this equivalence can be found in work of the second author
with Frenkel, Kazhdan and Vilonen~\cite{fgkv,fgv} and more recent work~\cite{gaitsgoryWhit}. 
Working in the context of $\mathcal{D}$-modules over a ground field of characteristic $0$
or $\ell$-adic sheaves over a ground field of arbitrary characteristic, in~\cite{fgv} the authors 
defined a candidate for the role of a Whittaker category on $\Gr$ using geometry of a 
complete curve and moduli stacks of bundles over it. A more direct, local definition 
of such a category is proposed in~\cite{gaitsgoryWhit}, where it is also shown that the two constructions 
produce equivalent categories; the methods of~\cite{gaitsgoryWhit} rely on recently developed 
techniques of $\infty$-categories.  Notice also that~\cite[Theorem 2.7.1(2)]{raskin} implies an 
equivalence between the above categories and the Iwahori--Whittaker category.

It was shown in~\cite{fgv} (see in particular~\cite[\S\S1.2.4--1.2.5]{fgv}) that their 
Whittaker category is a free right module over the monoidal category 
$\Perv_{\GO}(\Gr,\bk)$.
Thus, combining these works, we obtain another proof of the equivalence between
 $\Perv_{\IW}(\Gr,\bk)$ and $\Perv_{\GO}(\Gr,\bk) \cong \Rep(G^\vee_\bk)$, valid when we 
work with characteristic-$0$ coefficients. 

The above results cannot be automatically carried over to our present 
context, which is that of sheaves with coefficients of positive characteristic. 
However, the latter equivalence does generalize to our context, and amounts 
to our equivalence~\eqref{eqn:functor-equiv-intro}. Of course, we use different methods to prove it. 


As explained in~\cite[\S 1.1]{fgv}, in the case of characteristic-$0$ coefficients these properties are closely related to the Casselmann--Shalika formula, and in fact our proof uses the geometric counterpart to this formula known as the \emph{geometric Casselmann--Shalika formula}. (See also~\cite[\S 1.1.1]{ab} for the relation between the ``Whittaker'' and ``Iwahori--Whittaker'' conditions in the classical setting of modules over the affine Hecke algebra.)

\subsection{Application to tilting objects}

In Section~\ref{sec:applications} we provide a number of applications of this statement. An important one is concerned with the description of the \emph{tilting objects} in the Satake category. Namely, in the case when $\bk$ is a field of characteristic $\ell$, the \emph{tilting modules} (see e.g.~\cite[Chap.~E]{jantzen}) form an interesting family of objects in the category $\Rep(G^\vee_\bk)$. It is a natural question to try to characterize topologically the $\GO$-equivariant perverse sheaves on $\Gr$ corresponding to these objects. A first answer to this question was obtained by Juteau--Mautner--Williamson~\cite{jmw2}: they showed that, under some explicit conditions on $\ell$, the \emph{parity sheaves} on $\Gr$ for the stratification by $\GO$-orbits are perverse, and that their images under~\eqref{eqn:Satake} are the indecomposable tilting objects in $\Rep(G^\vee_\bk)$. This result was later extended by Mautner and the fourth author~\cite{mr} to the case when $\ell$ is good for $G$, and it played a crucial role in the proof (by Achar and the fifth author) of the Mirkovi\'c--Vilonen conjecture (or more precisely the corrected version of this conjecture suggested by Juteau~\cite{juteau}) on stalks of standard objects in the Satake category~\cite{achar-rider}.

It is known (see~\cite{jmw2}) that if $\ell$ is bad then the $\GO$-constructible parity sheaves on $\Gr$ are not necessarily perverse; so the answer to our question must be different in general. A conjecture was proposed by Juteau--Mautner--Williamson to cover this case, namely that the perverse cohomology objects of the parity complexes are tilting in $\Perv_{\GO}(\Gr,\bk)$ (so that all tilting objects are obtained by taking direct sums of direct summands of the objects obtained in this way). In our main application we confirm this conjecture, see Theorem~\ref{thm:parity-tilting}, hence obtain an answer to our question in full generality.

Using this description we prove a geometric analogue of a fundamental result for tilting modules, namely that these objects are preserved by tensor product and by restriction to a Levi subgroup. On the representation-theoretic side, the first proof of these results for general reductive groups is due to Mathieu~\cite{mathieu}, after partial results of Wang~\cite{wang-tilting} and Donkin~\cite{donkin}. A later proof was also found by Polo~\cite{polo} and Littelmann~\cite{littelmann} (with some restrictions) using Standard Monomial Theory. Finally, a more general statement (in terms of based modules for quantum groups) was obtained by Lusztig, see~\cite[\S 27.3]{lusztig-qg} (see also~\cite{paradowski,kaneda}).

In fact, combined with the Satake equivalence, our proof can also be considered as providing a new complete proof of these properties of tilting modules. In~\cite{br}, Baumann and the fourth author also use these facts to obtain a slight simplification of the proof of the geometric Satake equivalence. (Note that the proofs in the present paper do not rely on the latter result.)

%
%
%

\subsection{Acknowledgements}

The final stages of this work were accomplished while the fourth author was a fellow of the Freiburg Institute for Advanced Studies, as part of the Research Focus ``Cohomology in Algebraic Geometry and Representation Theory'' led by A. Huber--Klawitter, S. Kebekus and W. Soergel.

We thank P.~Achar and G.~Williamson for useful discussions on the subject of this paper, and the referees for their helpful comments.

\section{Constructible sheaves on affine Grassmannians and affine flag varieties}
\label{sec:cont-sheaves}


\subsection{Notation}
\label{ss:notation}

Let $\F$ be an algebraically closed field of characteristic $p>0$. Let $G$ be a connected reductive algebraic group over $\F$, let $B^- \subset G$ be a Borel subgroup, and let $T \subset B^-$ be a maximal torus. Let also $B^+ \subset G$ be the Borel subgroup opposite to $B^-$ (with respect to $T$), and let $U^+$ be its unipotent radical.

We denote by $\bX:=X^*(T)$ the character lattice of $T$, by $\bX^\vee := X_*(T)$ its cocharacter lattice, by $\Delta \subset \bX$ the root system of $(G,T)$, and by $\Delta^\vee \subset \bX^\vee$ the corresponding coroots. We choose the system of positive roots $\Delta^+ \subset \Delta$ consisting of the $T$-weights in $\mathrm{Lie}(U^+)$, and denote by $\bX^\vee_+ \subset \bX^\vee$, resp.~$\bX^\vee_{++} \subset \bX^\vee$ the corresponding subset of dominant cocharacters, resp.~of strictly dominant cocharacters. We also denote by $\Delta_{\mathrm{s}} \subset \Delta$ the corresponding subset of simple roots, and set
\[
 \rho = \frac{1}{2} \sum_{\alpha \in \Delta^+} \alpha \quad \in \mathbb{Q} \otimes_{\mathbb{Z}} \bX.
\]
 
 For any $\alpha \in \Delta_{\mathrm{s}}$ we choose an isomorphism between the additive group $\Ga$ and the root subgroup $U_\alpha$ of $G$ associated with $\alpha$, and denote it $u_\alpha$.
 
 We will assume\footnote{This assumption holds in particular if $G$ is semisimple of adjoint type.} that there exists $\varsigma \in \bX^\vee$ such that $\langle \varsigma, \alpha \rangle = 1$ for any $\alpha \in \Delta_{\mathrm{s}}$; then we have $\bX^\vee_{++} = \bX^\vee_+ + \varsigma$. (Such a cocharacter might not be unique; we fix a choice once and for all.)
 
 Let $\Wf$ be the Weyl group of $(G,T)$, and let $W:=\Wf \ltimes \bX^\vee$ be the corresponding (extended) affine Weyl group. For $\lambda \in \bX^\vee$ we will denote by $t_\lambda$ the associated element of $W$. If $w \in W$ and $w = t_\lambda v$ with $\lambda \in \bX^\vee$ and $v \in \Wf$, we set
%
 \[
  \ell(w) = \sum_{\substack{\alpha \in \Delta^+ \\ v(\alpha) \in \Delta^+}} |\langle \lambda, \alpha \rangle| + \sum_{\substack{\alpha \in \Delta^+ \\ v(\alpha) \in -\Delta^+}} |1+\langle \lambda, \alpha \rangle|.
 \]
 Then the restriction of $\ell$ to the semi-direct product $W^{\mathrm{Cox}}$ of $\Wf$ with the coroot lattice is the length function for a natural Coxeter group structure, and if we set $\Omega := \{w \in W \mid \ell(w)=0\}$ then multiplication induces a group isomorphism 
 \[
 W^{\mathrm{Cox}} \rtimes \Omega \xrightarrow{\sim} W.
\]

\subsection{The affine Grassmannian and the affine flag variety}
\label{ss:Gr-Fl}

For the facts we state here, we refer to~\cite{faltings}.

We set $\scK:=\F( \hspace{-0.5pt} (z) \hspace{-0.5pt} )$, $\scO := \F[ \hspace{-0.5pt} [z] \hspace{-0.5pt} ]$, and consider the ind-group scheme $G_\scK$ (denoted $LG$ in~\cite{faltings}) and its group subscheme $G_\scO$ (denoted $L^+ G$ in~\cite{faltings}).
We denote by $I^- \subset G_\scO$ the Iwahori subgroup associated with $B^-$, i.e.~the inverse image of $B^-$ under the morphism $G_\scO \to G$ sending $z$ to $0$. We consider the affine Grassmannian $\Gr$ and the affine flag variety $\Fl$ defined by
\[
\Gr := G_\scK/G_\scO, \quad \Fl := G_\scK / I^-.
\]
We denote by $\pi : \Fl \to \Gr$ the projection morphism.

Any $\lambda \in \bX^\vee$ defines a point $z^\lambda \in T_\scK \subset G_\scK$, hence a point $L_\lambda := z^{\lambda} \GO \in \Gr$. We set
\[
\Gr^\lambda := G_\scO \cdot L_\lambda.
\]
Then $\Gr^\lambda$ only depends on the $\Wf$-orbit of $\lambda$. Moreover,
the Bruhat decomposition implies that
\[
\Gr = 
\bigsqcup_{\lambda \in \bX^\vee_+} \Gr^\lambda.
\]
We will denote by $j_\lambda : \Gr^\lambda \to \Gr$ the embedding.

For $\lambda \in \bX^\vee_+$, we will denote by $P_\lambda \subset G$ the parabolic subgroup of $G$ containing $B^-$ associated with the subset of $\Delta_{\mathrm{s}}$ consisting of those simple roots which are orthogonal to $\lambda$. Then $P_\lambda$ is the stabilizer of $L_\lambda$ in $G$, so that we have a canonical isomorphism $G/P_\lambda \xrightarrow{\sim} G \cdot L_\lambda$. Under this identification, it is known that the map $p_\lambda : \Gr^\lambda \to G/P_\lambda$ sending $x$ to $\lim_{t \to 0} t \cdot x$ (where we consider the $\Gm$-action on $\Gr$ via loop rotation) is a morphism of algebraic varieties, and realizes $\Gr^\lambda$ as an affine bundle over $G/P_\lambda$ (see e.g.~\cite[Lemme~2.3]{ngo-polo}).


%
%

It is well known (see e.g.~\cite{lusztig} or~\cite[\S 2]{ngo-polo}) that if $\lambda \in \bX^\vee_+$, then we have
\[
\dim(\Gr^\lambda) = \langle \lambda, 2\rho \rangle = \sum_{\alpha \in \Delta^+} \langle \lambda, \alpha \rangle.
\]
We denote by $\preceq$ the order on $\bX^\vee_+$ determined by
\[
\lambda \preceq \mu \quad \text{ iff $\mu - \lambda$ is a sum of positive coroots.}
\]
Then for $\lambda, \mu \in \bX^\vee_+$ we have
\[
\Gr^\lambda \subset \overline{\Gr^\mu} \quad \text{iff} \quad \lambda \preceq \mu.
\]


\subsection{Some categories of sheaves on \texorpdfstring{$\Gr$}{Gr} and \texorpdfstring{$\Fl$}{Fl}}
\label{ss:categories}

We let $\ell$ be a prime number which is different from $p$, and let $\bk$ be either a finite extension of $\Ql$, or the ring of integers in such an extension, or a finite field of characteristic $\ell$. In this paper we will be concerned with the constructible derived categories $\Db_c(\Gr,\bk)$ and $\Db_c(\Fl,\bk)$ of \'etale $\bk$-sheaves on $\Gr$ and $\Fl$, respectively.
If $K \subset \GO$ is a subgroup, we will also denote by $\Db_K(\Gr,\bk)$ and $\Db_K(\Fl,\bk)$ the (constructible) $K$-equivariant derived category of $\bk$-sheaves on $\Gr$ and $\Fl$, in the sense of Bernstein--Lunts~\cite{bernstein-lunts}. Each of these categories is endowed with the perverse t-structure, whose heart will be denoted $\Perv(\Gr,\bk)$, $\Perv(\Fl,\bk)$, $\Perv_K(\Gr,\bk)$ and $\Perv_K(\Fl,\bk)$ respectively.

\begin{rmk}\phantomsection
\label{rmk:definition-categories}
\begin{enumerate}
\item
Since $\Gr$ and $\Fl$ are ind-varieties and not varieties, the definition of the categories considered above requires some care; 
see e.g.~\cite[\S 2.2]{nadler} or~\cite[Appendix]{gaitsgory} for details. 
We will not mention this point in the body of the paper, and simply refer to objects in these categories as complexes of sheaves.
\item
\label{it:Whittaker-analytic}
Recall that by \cite{mv} the category $\Rep(G^\vee_R)$ of algebraic representations of the
group scheme $G^\vee_R$ over any Noetherian commutative base ring $R$ of finite global dimension is equivalent to the corresponding
category of spherical perverse sheaves on $\Gr_{\mathbb{C}}$ in its analytic topology. 
More restrictive assumptions on the base ring in the present paper come from our need to 
use the Artin--Schreier sheaf (see~\S\ref{ss:def-IW}), which is only defined in the context of \'etale sheaves over a variety 
in positive characteristic; this setting yields categories of sheaves with coefficients
as above. Notice however that some constructions involving the Artin--Schreier
sheaf do have an analogue for constructible sheaves in the classical topology (see e.g.~\cite{wang} for the example of Fourier--Deligne transform). We expect that such a
counterpart of the Whittaker category can also be defined (see \cite[Remark 10.3.6]{ag}
for a possible approach); this would allow one to extend our main result
to more general coefficient rings.
\end{enumerate}
\end{rmk}

If $K' \subset K \subset \GO$ are subgroups, we will denote by
\[
 \For^{K}_{K'} : \Db_K(\Gr,\bk) \to \Db_{K'}(\Gr,\bk), \quad \For^{K}_{K'} : \Db_K(\Fl,\bk) \to \Db_{K'}(\Fl,\bk)
\]
the natural forgetful functors. If $K/K'$ is of finite type, these functors have both a right and a left adjoint, which will be denoted $\rInd_{K'}^K$ and $\lInd_{K'}^K$ respectively. If we write $X$ for $\Gr$ or $\Fl$, these functors can be described explicity by
\[
 \rInd_{K'}^K(\cF) = a_* \bigl( \underline{\bk} \, \widetilde{\boxtimes} \, \cF \bigr) \quad \text{and} \quad \lInd_{K'}^K(\cF) = a_! \bigl( \underline{\bk} \, \widetilde{\boxtimes} \, \cF \bigr)[2(\dim K/ K')],
\]
where $\underline{\bk} \, \widetilde{\boxtimes} \, (-)$ is the functor sending an object $\cF$ to the only object in $\Db_K(K \times^{K'} X, \bk)$ whose pullback to $K \times X$ (an object of $\Db_{K \times K'}(K \times X,\bk)$, where $K'$ acts on $K \times X$ via $h \cdot (g,x)=(gh^{-1}, h \cdot x)$) is isomorphic to $\underline{\bk}_K \lboxtimes_\bk \cF$.
When $K'=\{1\}$ we will write $\For_K$ for $\For^K_{\{1\}}$.

\subsection{Convolution}
\label{ss:convolution}

We will make extensive use of the convolution construction, defined as follows. Consider $\cF,\cG$ in $\Db_{\GO}(\Gr,\bk)$, and the diagram
\[
 \Gr \times \Gr \xleftarrow{p^\Gr} \GK \times \Gr \xrightarrow{q^\Gr} \GK \times^{\GO} \Gr \xrightarrow{m^\Gr} \Gr,
\]
where $p^\Gr$ and $q^\Gr$ are the quotient morphisms, and $m^\Gr$ is induced by the $\GK$-action on $\Gr$. Consider the action of $\GO \times \GO$ on $\GK \times \Gr$ defined by
\[
 (g_1,g_2) \cdot (h_1,h_2\GO) = (g_1h_1(g_2)^{-1}, g_2 h_2\GO).
\]
Then the functor $(q^\Gr)^*$ induces an equivalence of categories
\[
 \Db_{\GO}(\GK \times^{\GO} \Gr,\bk) \xrightarrow{\sim} \Db_{\GO \times \GO}(\GK \times \Gr,\bk).
\]
Hence there exists a unique object $\cF \, \widetilde{\boxtimes} \, \cG$ such that
\[
 (q^\Gr)^* \bigl( \cF \, \widetilde{\boxtimes} \, \cG \bigr) = (p^\Gr)^* \bigl( \cF \lboxtimes_{\bk} \cG \bigr).
\]
Then the convolution product of $\cF$ and $\cG$ is defined by
\begin{equation}
\label{eqn:formula-convolution}
\cF \star^{\GO} \cG := (m^\Gr)_* \bigl( \cF \, \widetilde{\boxtimes} \, \cG \bigr).
\end{equation}
This construction endows the category $\Db_{\GO}(\Gr,\bk)$ with the structure of a monoidal category. A similar formula defines a right action of this monoidal category on $\Db_K(\Gr,\bk)$, for any $K \subset \GO$. (This action will again be denoted $\star^{\GO}$.)

\begin{rmk}
 Note that if $\bk$ is not a field, the convolution product considered above is \emph{not} same as the one considered (when $\cF$ and $\cG$ are perverse sheaves) in~\cite{mv}: the product considered in~\cite{mv} is rather defined as $\pH^0(\cF \star^{\GO} \cG)$ in our notation.
\end{rmk}

\begin{lem}
\label{lem:convolution-exact}
 Assume that $\bk$ is a field. If $\cF$ belongs to $\Perv(\Gr,\bk)$ and $\cG$ belongs to $\Perv_{\GO}(\Gr,\bk)$, then $\cF \star^{\GO} \cG$ belongs to $\Perv(\Gr,\bk)$.
\end{lem}

\begin{proof}
 This claim follows from the description of convolution in terms of nearby cycles obtained in~\cite[Proposition~6]{gaitsgory}. (In~\cite{gaitsgory}, only the case of characteristic-$0$ coefficients is treated. However the same proof applies in general, simply replacing~\cite[Proposition~1]{gaitsgory} by~\cite[Proposition~2.2]{mv}.) An earlier and different proof is also suggested in~\cite[Remark~12.11]{fm}.
\end{proof}

\begin{rmk}
\begin{enumerate}
 \item 
 The description of convolution in terms of nearby cycles as in~\cite{gaitsgory} works for general coefficients (if convolution is defined as in~\eqref{eqn:formula-convolution}). The nearby cycles functor is t-exact in this generality, but this description also involves a (derived) external tensor product. If $\bk$ is not a field this tensor product operation is not t-exact, which explains the failure of Lemma~\ref{lem:convolution-exact} in this setting.
 \item
 In case $\cF$ is $\GO$-equivariant, the fact that $\cF \star^{\GO} \cG$ is perverse can also be deduced from the stratified semismallness of the convolution diagram. The stratified semismallness property follows from Lusztig's results in~\cite{lusztig} (see~\cite[Remark 1.6.5(2)]{br} for details); another proof using semi-infinite orbits is explained
in the later work~\cite{mv}.
 \end{enumerate}
\end{rmk}

A very similar construction as the one considered above, based on the diagram
\[
 \Fl \times \Fl \xleftarrow{p^\Fl} \GK \times \Fl \xrightarrow{q^\Fl} \GK \times^{I^-} \Fl \xrightarrow{m^\Fl} \Fl,
\]
provides a convolution product $\star^{I^-}$ on $\Db_{I^-}(\Fl,\bk)$, which endows this category with the structure of a monoidal category, and defines a right action of this monoidal category on $\Db_K(\Fl,\bk)$, for any $K \subset \GO$.
Again the same formulas, using the diagram
\[
 \Fl \times \Gr \xleftarrow{p^{\Fl}_{\Gr}} \GK \times \Gr \xrightarrow{q^{\Fl}_{\Gr}} \GK \times^{I^-} \Gr \xrightarrow{m^{\Fl}_{\Gr}} \Gr,
\]
allows to define a bifunctor
\[
 \Db_K(\Fl,\bk) \times \Db_{I^-}(\Gr,\bk) \to \Db_K(\Gr,\bk),
\]
which will once again be denoted $\star^{I^-}$.

The following lemma is standard; its proof is left to interested readers.

\begin{lem}
\label{lem:convolution-For}
 For any subgroup $K \subset \GO$, any $\cF$ in $\Db_K(\Fl,\bk)$ and any $\cG$ in $\Db_{\GO}(\Gr,\bk)$, there exists a canonical isomorphism
 \[
  \cF \star^{I^-} \For^{\GO}_{I^-}(\cG) \cong \pi_*(\cF) \star^{\GO} \cG
 \]
 in $\Db_K(\Gr,\bk)$.
\end{lem}

In the following lemma we consider the convolution bifunctor
\[
 (-) \star^{\GO} (-) : \Db_K(\Gr,\bk) \times \Db_{\GO}(\Fl,\bk) \to \Db_K(\Fl,\bk)
\]
(constructed once again using formulas similar to those above). Its proof is easy, and
left to the reader.

\begin{lem}
\label{lem:convolution-induction}
 Let $\cF$ in $\Db_K(\Gr,\bk)$ and $\cG$ in $\Db_{I^-}(\Fl,\bk)$. Then there exists a canonical isomorphism
 \[
  \pi^*(\cF) \star^{I^-} \cG \cong \cF \star^{\GO} \rInd_{I^-}^{\GO}(\cG)
 \]
 in $\Db_K(\Fl,\bk)$.
\end{lem}

\section{Spherical vs. Iwahori--Whittaker}
\label{sec:spherical-IW}

\subsection{Equivariant perverse sheaves on \texorpdfstring{$\Gr$}{Gr}}
\label{ss:categories-perv}


For $\lambda \in \bX^+$, we will denote by
\[
\cJ_!(\lambda,\bk):= \pH^0 \bigl( (j_\lambda)_! \underline{\bk}_{\Gr^\lambda}[\langle \lambda, 2\rho \rangle] \bigr), \quad \text{resp.} \quad \cJ_*(\lambda,\bk):= \pH^0 \bigl( (j_\lambda)_* \underline{\bk}_{\Gr^\lambda}[\langle \lambda, 2\rho \rangle] \bigr),
\]
the standard, resp.~costandard, $G_\scO$-equivariant perverse sheaf on $\Gr$ associated with $\lambda$.
We will also denote by $\cJ_{!*}(\lambda,\bk)$ the image of any generator of the free rank-$1$ $\bk$-module
\[
\Hom_{\Perv_{G_\scO}(\Gr, \bk)}(\cJ_!(\lambda,\bk), \cJ_*(\lambda,\bk)).
\]
If $\bk$ is a field then $\cJ_{!*}(\lambda,\bk)$ is a simple perverse sheaf, which is both the head of $\cJ_!(\lambda,\bk)$ and the socle of $\cJ_*(\lambda,\bk)$.

Recall the notion of highest weight category, whose definition is spelled out e.g.~in~\cite[Definition~7.1]{riche-hab}. (These conditions are obvious extensions of those considered in~\cite[\S 3.2]{bgs}, which were preceded by a related study in~\cite{cps}.)

\begin{lem}
\label{lem:Perv-GO}
Assume that $\bk$ is a field.
The category $\Perv_{G_\scO}(\Gr, \bk)$ is a highest weight category with weight poset $(\bX^\vee_+, \preceq)$, standard objects $\{\cJ_!(\lambda,\bk) : \lambda \in \bX^\vee_+\}$, and costandard objects $\{\cJ_*(\lambda,\bk) : \lambda \in \bX^\vee_+\}$. Moreover, if $\mathrm{char}(\bk)=0$ then this category is semisimple.
\end{lem}

\begin{proof}
The first claim is an easy consequence of~\cite[Proposition~10.1(b)]{mv}; see~\cite[Proposition~1.12.4]{br} for details.
If $\mathrm{char}(\bk)=0$, the semisimplicity of the category $\Perv_{G_\scO}(\Gr, \bk)$ is well known: see~\cite[Proposition~1]{gaitsgory} (or~\cite[\S 1.4]{br} for an expanded version).
\end{proof}

\begin{rmk}\phantomsection
\label{rmk:Hom-J!-J*}
\begin{enumerate}
\item
\label{it:semisimple-J*-J!}
If $\bk$ is a field of characteristic $0$, the semisimplicity of the category $\Perv_{G_\scO}(\Gr, \bk)$ implies in particular that the natural maps $\cJ_!(\lambda,\bk) \to \cJ_{!*}(\lambda,\bk) \to \cJ_*(\lambda,\bk)$ are isomorphisms.
\item
\label{it:Hom-J!-J*}
 For any coefficients $\bk$, we have
 \[
  \Hom_{\Db \Perv_{G_\scO}(\Gr,\bk)} \bigl( \cJ_!(\lambda,\bk), \cJ_*(\mu,\bk)[n] \bigr)=
  \begin{cases}
   \bk & \text{if $n=0$ and $\lambda=\mu$;} \\
    0 & \text{otherwise.}
\end{cases}
 \]
In fact, to prove this it suffices to prove the similar claim for perverse sheaves on $Z$, where $Z \subset \Gr$ is any closed finite union of $G_\scO$-orbits containing $\Gr^\lambda$ and $\Gr^\mu$. In the case $\bk$ is a field, this claim is a consequence of Lemma~\ref{lem:Perv-GO} (or rather its version for $Z$). The case when $\bk$ is the ring of integers in a finite extension of $\Ql$ follows. Indeed, since the category $\Perv_{G_\scO}(Z,\bk)$ has enough projectives we can consider the complex of $\bk$-modules
\[
 M=R\Hom_{\Perv_{G_\scO}(Z,\bk)}(\cJ_!(\lambda,\bk), \cJ_*(\mu,\bk)).
\]
If $\bk_0$ is the residue field of $\bk$, it is not difficult (using the results of~\cite[\S 8 and \S 10]{mv}, and in particular the fact that $\bk_0 \lotimes_\bk \cJ_?(\lambda,\bk) \cong \cJ_?(\lambda,\bk_0)$ for $? \in \{!,*\}$, see~\cite[Proposition~8.1]{mv}) to check that
\[
 \bk_0 \lotimes_\bk M \cong R\Hom_{\Perv_{G_\scO}(Z,\bk_0)}(\cJ_!(\lambda,\bk_0), \cJ_*(\mu,\bk_0)).
\]
We deduce that the left-hand side is isomorphic to $\bk_0$ in the derived category of $\bk_0$-vector spaces; this implies that $M$ is isomorphic to $\bk$ in the derived category of $\bk$-modules.
\end{enumerate}
\end{rmk}

In Section~\ref{sec:applications} we will also encounter some $I^-_\uu$-equivariant perverse sheaves on $\Gr$, where $I^-_\uu$ is the pro-unipotent radical of $I^-$. In particular, we have
\[
\Gr = \bigsqcup_{\mu \in \bX^\vee} I_\uu^- \cdot L_\mu,
\]
and we will denote by $\Delta^\Gr_\mu(\bk)$, resp.~$\nabla^\Gr_\mu(\bk)$, the standard, resp.~costandard, perverse sheaf associated with $\mu$, i.e.~the $!$-direct image (resp.~$*$-direct image) of the constant perverse sheaf of rank $1$ on $I^-_\uu \cdot L_\mu$. (These objects are perverse sheaves thanks to~\cite[Corollaire~4.1.3]{bbd}, because the orbits $I_\uu^- \cdot L_\mu$ are affine spaces.)

\subsection{Category of Iwahori--Whittaker perverse sheaves}
\label{ss:def-IW}

We now denote by 
$I^+ \subset G_\scO$ the Iwahori subgroup associated with $B^+$. We also denote by $I^+_\uu$ the pro-unipotent radical of $I^+$, i.e.~the inverse image of $U^+$ under the map $I^+ \to B^+$.

We assume that there exists a primitive $p$-th root of unity in $\bk$, and fix one. This choice determines a
character $\psi$ of the prime subfield of $\F$ (with values in $\bk^\times$), and we denote by $\cL^\bk_\psi$ the corresponding Artin--Schreier local system on $\Ga$. (Below, some arguments using Verdier duality will also involve the Artin--Schreier local system $\cL^\bk_{-\psi}$ associated with the character $\psi^{-1}$; clearly these two versions play similar roles.) We also consider the ``generic'' character $\chi : U^+ \to \Ga$ 
defined as the composition
\[
 U^+ \twoheadrightarrow U^+ / [U^+,U^+] \xleftarrow[\sim]{\prod_{\alpha} u_\alpha} \prod_{\alpha \in \Delta_{\mathrm{s}}} \Ga \xrightarrow{+} \Ga,
\]
and denote by $\chi_{I^+}$ its composition with the projection $I^+_\uu \twoheadrightarrow U^+$. We can then define the ``Iwahori--Whittaker'' derived category
\[
\Db_{\IW}(\Gr, \bk)
\]
as the $(I^+_\uu, \chi_{I^+}^*(\cL^\bk_\psi))$-equivariant constructible derived category of $\bk$-sheaves on $\Gr$ (see e.g.~\cite[Appendix~A]{modrap1} for a review of the construction of this category). This category admits a perverse t-structure, whose heart will be denoted $\Perv_{\IW}(\Gr,\bk)$, and moreover the ``realization functor''
\[
\Db \Perv_{\IW}(\Gr,\bk) \to \Db_{\IW}(\Gr, \bk)
\]
is an equivalence of triangulated categories.

Note that any ring of coefficients considered above appears in an $\ell$-modular triple $(\K,\O,\L)$ where $\K$ is a finite extension of $\Ql$, $\O$ is its ring of integers, and $\L$ is the residue field of $\O$. In this setting the embedding $\O \hookrightarrow \K$ and the projection $\O \to \L$ induce bijections between the $p$-th roots of unity in $\O$, $\K$ and $\L$. Therefore, choosing a primitive root in any of these rings provides primitive roots in all three rings, and we can then consider extension of scalars functors
\[
 \K \lotimes_\O (-) : \Db_\IW(\Gr,\O) \to \Db_\IW(\Gr,\K), \quad \L \lotimes_\O (-) : \Db_\IW(\Gr,\O) \to \Db_\IW(\Gr,\L).
\]
These functors will play a crucial role in our arguments below. (Here the first functor is t-exact, but the second one is only right t-exact.)

For $\lambda \in \bX^\vee$ we set
\[
X_\lambda := I^+ \cdot L_\lambda.
\]
Then again we have
\[
\Gr = \bigsqcup_{\lambda \in \bX^\vee} X_\lambda.
\]

\begin{lem}
\label{lem:orbits-IW}
The orbit $X_\lambda$ supports an $(I^+_\uu, \chi_{I^+}^*(\cL^\bk_\psi))$-equivariant local system iff $\lambda \in \bX^\vee_{++}$.
\end{lem}

\begin{proof}[Sketch of proof]
Let $\lambda \in \bX^\vee_+$, and consider the affine bundle $\Gr^\lambda \to G/P_\lambda$ (see~\S\ref{ss:Gr-Fl}).
The decomposition of $\Gr^\lambda$ in $I_\uu^+$-orbits is obtained by pullback from the decomposition of $G/P_\lambda$ into $U^+$-orbits; in particular, for $\mu \in \Wf \cdot \lambda$, $X_\mu$ supports an $(I^+_\uu, \chi_{I^+}^*(\cL^\bk_\psi))$-equivariant local system iff its image in $G/P_\lambda$ is a free $U^+$-orbit. If $\lambda \notin \bX^\vee_{++}$ there is no such orbit in $G/P_\lambda$, and if $\lambda \in \bX^\vee_{++}$ there is exactly one, corresponding to $X_\lambda$.
\end{proof}

Note that if $\lambda \in \bX^\vee_{++}$, since $X_\lambda$ is open dense in $\Gr^\lambda$ we have
\begin{equation}
\label{eqn:dim-Xlambda}
\dim(X_\lambda) = \langle \lambda, 2\rho \rangle
\end{equation}
and
\[
X_\lambda \subset \overline{X_\mu} \quad \text{iff} \quad \lambda \preceq \mu.
\]
For $\lambda \in \bX^\vee_{++}$ we will denote by
\[
\Delta^{\IW}_\lambda(\bk), \quad \text{resp.} \quad \nabla^{\IW}_\lambda(\bk),
\]
the standard, resp.~costandard, $(I^+_\uu, \chi_{I^+}^*(\cL^\bk_\psi))$-equivariant perverse sheaf on $\Gr$ associated with $\lambda$, i.e.~the $!$-extension, resp.~$*$-extension, to $\Gr$ of the free rank-$1$ $(I^+_\uu, \chi_{I^+}^*(\cL^\bk_\psi))$-equivariant perverse sheaf on $X_\lambda$. (Once again, these objects are perverse sheaves thanks to~\cite[Corollaire~4.1.3]{bbd}.) We will also denote by $\IC_\lambda^\IW(\bk)$ the image of any generator of the rank-$1$ free $\bk$-module
\[
\Hom_{\Perv_{\IW}(\Gr, \bk)}(\Delta_\lambda^\IW(\bk), \nabla_\lambda^\IW(\bk)).
\]
If $\bk$ is a field then $\IC_\lambda^\IW(\bk)$ is a simple perverse sheaf.

Note that since $\varsigma$ is minimal in $\bX^\vee_{++}$ for $\preceq$, we have
\begin{equation}
\label{eqn:Delta-nabla-sigma}
\Delta^\IW_{\varsigma}(\bk) = \nabla^\IW_\varsigma(\bk) = \IC^\IW_\varsigma(\bk).
\end{equation}

\begin{lem}
\label{lem:stalks-IC-IW}
Assume that $\bk$ is a field of characteristic $0$. Then the $i$-th cohomology of the stalks of $\IC^\IW_\lambda(\bk)$ vanish unless $i \equiv \dim X_\lambda \pmod 2$.
\end{lem}

\begin{proof}[Sketch of proof]
Since the morphism $\pi$ is smooth, by standard properties of perverse sheaves (see e.g.~\cite[\S 4.2.6]{bbd}) it suffices to prove a similar statement on $\Fl$ instead of $\Gr$. Now the Decomposition Theorem implies that all simple $(I_\uu^+, \chi_{I^+}^*(\cL^\bk_\psi))$-equivariant perverse sheaves on $\Fl$ can be obtained from the one corresponding to the orbit of the base point by convolving on the right with $I^-$-equivariant simple perverse sheaves on $\Fl$ corresponding to orbits of dimension either $0$ or $1$. Standard arguments (going back at least to~\cite{springer}) show that these operations preserve the parity-vanishing property of stalks, and the claim follows.
\end{proof}

\begin{rmk}
\label{rmk:parity}
Assume that $\bk$ is a field.
Following~\cite{jmw}\footnote{In~\cite{jmw} the authors consider the setting of ``ordinary'' constructible complexes. However, as observed already in~\cite[\S 11.1]{tilting} or~\cite[\S 6.2]{mkdkm}, their considerations apply verbatim in our Iwahori--Whittaker setting.} we will say that an object of $\Db_{\IW}(\Gr, \bk)$ is \emph{even}, resp.~\emph{odd}, if its restriction and corestriction to each stratum is concentrated in even, resp.~odd, degrees, and that it is \emph{parity} if it is isomorphic to a direct sum $\mathcal{F} \oplus \mathcal{F}'$ with $\mathcal{F}$ even and $\mathcal{F}'$ odd. Using this language, Lemma~\ref{lem:stalks-IC-IW} states that if $\mathrm{char}(\bk)=0$ then the objects $\IC^{\IW}_\lambda(\bk)$ are parity, of the same parity as $\dim(X_\lambda)$.
\end{rmk}


\begin{cor}
\label{cor:Perv-IW}
Assume that $\bk$ is a field.
The category $\Perv_{\IW}(\Gr, \bk)$ is a highest weight category with weight poset $(\bX^\vee_{++}, \preceq)$, standard objects $\{\Delta_\lambda^\IW(\bk) : \lambda \in \bX^\vee_{++}\}$, and costandard objects $\{\nabla^\IW_\lambda(\bk) : \lambda \in \bX^\vee_{++}\}$. Moreover, if $\mathrm{char}(\bk)=0$ then this category is semisimple.
\end{cor}

\begin{proof}
The first claim is standard, as e.g.~in~\cite[\S 3.3]{bgs}. For the second claim, we observe that the orbits $X_\lambda$ (for $\lambda \in \bX^\vee_{++}$) have dimensions of constant parity on each connected component of $\Gr$, see~\eqref{eqn:dim-Xlambda}. 
Using this and Lemma~\ref{lem:stalks-IC-IW},
the semisimplicity can be proved exactly as in the case of the category $\Perv_{G_\scO}(\Gr, \bk)$. Namely, we have to prove that
\[
\Ext^1_{\Perv_{\IW}(\Gr, \bk)}(\IC^\IW_\lambda(\bk),\IC^\IW_\mu(\bk))=\Hom_{\Db_{\IW}(\Gr, \bk)}(\IC^\IW_\lambda(\bk),\IC^\IW_\mu(\bk)[1])
\]
vanishes for any $\lambda,\mu$. If $X_\lambda$ and $X_\mu$ belong to different connected components of $\Gr$ then this claim is obvious; otherwise $\IC^\IW_\lambda(\bk)$ and $\IC^\IW_\mu(\bk)$ are either both even or both odd (see Remark~\ref{rmk:parity}), so that the desired vanishing follows from~\cite[Corollary~2.8]{jmw}.
\end{proof}

\begin{rmk}\phantomsection
\label{rmk:Hom-J!-J*-IW}
\begin{enumerate}
\item
\label{it:J!-J*-IW-simple}
Once Corollary~\ref{cor:Perv-IW} is known, one can refine Lemma~\ref{lem:stalks-IC-IW} drastically: if $\bk$ is a field of characteristic $0$, then the simple perverse sheaves $\IC^{\IW}_\lambda(\bk)$ are clean, in the sense that if $i_\mu : X_\mu \to \Gr$ is the embedding, for any $\mu \neq \lambda$ we have
\begin{equation}
\label{eqn:IC-IW-clean}
i_\mu^* \bigl( \IC^{\IW}_\lambda(\bk) \bigr) = i_\mu^! \bigl( \IC^{\IW}_\lambda(\bk) \bigr)=0.
\end{equation}
In fact, as in Remark~\ref{rmk:Hom-J!-J*}\eqref{it:semisimple-J*-J!}, the semisimplicity claim in Corollary~\ref{cor:Perv-IW} implies that the natural maps
\[
 \Delta^{\IW}_\lambda(\bk) \to \IC^{\IW}_\lambda(\bk) \to \nabla^{\IW}_\lambda(\bk)
\]
are isomorphisms, which is equivalent to~\eqref{eqn:IC-IW-clean}.
(See also~\cite[Corollary~2.2.3]{abbgm} for a different proof of~\eqref{eqn:IC-IW-clean}.)
This observation can be used to give a new proof of the main result of~\cite{fgv}, hence of the geometric Casselman--Shalika formula.

\item
\label{it:Hom-J!-J*-IW}
The same arguments as in Remark~\ref{rmk:Hom-J!-J*}\eqref{it:Hom-J!-J*} show that for any coefficients $\bk$, any $\lambda,\mu \in \bX^\vee_{++}$ and any $n \in \Z$ we have
 \[
  \Hom_{\Db \Perv_{\IW}(\Gr, \bk)} \left( \Delta_\lambda^\IW(\bk), \nabla_\mu^\IW(\bk)[n] \right) =
  \begin{cases}
   \bk & \text{if $\lambda=\mu$ and $n=0$;} \\
   0 & \text{otherwise.}
  \end{cases}
 \]
 (In this case, the existence of enough projectives in $\Perv_{\IW}(Z,\bk)$ can be checked using the techniques of~\cite[\S 2]{rsw}.)
\end{enumerate}
\end{rmk}

\subsection{Statement}
\label{ss:equivalence}

We consider the functor
\[
\Phi : \Db_{G_\scO}(\Gr, \bk) \to \Db_{\IW}(\Gr, \bk)
\]
defined by
\[
\Phi(\cF) = \Delta^\IW_\varsigma(\bk) \star^{\GO} \cF.
\]
In view of~\eqref{eqn:Delta-nabla-sigma} (or, alternatively, arguing as in~\cite{bbrm}), in this definition $
\Delta^\IW_{\varsigma}(\bk)$ can be replaced by $\nabla^\IW_\varsigma(\bk)$ or $ \IC^\IW_\varsigma(\bk)$; in particular this shows that the conjugate of $\Phi$ by Verdier duality is the similar functor using the character $\psi^{-1}$ instead of $\psi$. 

\begin{lem}
\label{lem:Phi-exact}
 The functor $\Phi$ is t-exact for the perverse t-structures.
\end{lem}

\begin{proof}
 In the case where $\bk$ is a field, the claim follows from Lemma~\ref{lem:convolution-exact}. The general case follows using an $\ell$-modular triple $(\K,\O,\L)$ as in~\S\ref{ss:def-IW} with $\bk=\O$, and the associated extension of scalars functors. Namely, 
 if $\cF$ is in $\Perv_{\GO}(\Gr,\O)$, then
 \[
 \K \lotimes_\O \Phi(\cF) \cong \Phi(\K \lotimes_\O \cF)
 \]
 is perverse; hence any perverse cohomology object $\pH^i(\Phi(\cF))$ with $i \neq 0$ is torsion. On the other hand,
 \[
 \L \lotimes_\O \Phi(\cF) \cong \Phi(\L \lotimes_\O \cF)
 \]
 lives in perverse degrees $-1$ and $0$ since $\L \lotimes_\O \cF$ lives in these degrees. If $\pH^i(\Phi(\cF))$ were nonzero for some $i>0$, then taking $i$ maximal with this property we would obtain that $\pH^i( \L \lotimes_\O \Phi(\cF)) \neq 0$, a contradiction. On the other hand, if $\pH^i(\Phi(\cF))$ was nonzero for some $i<0$, then taking $i$ minimal with this property we would obtain that $\pH^{i-1}( \L \lotimes_\O \Phi(\cF)) \neq 0$, a contradiction again.
 \end{proof}

We will denote by
\[
 \Phi^0 : \Perv_{G_\scO}(\Gr, \bk) \to \Perv_{\IW}(\Gr, \bk)
\]
the restriction of $\Phi$ to the hearts of the perverse t-structures, so that $\Phi^0$ is an exact functor between abelian categories.
The main result of this section is the following theorem, whose proof will be given in the next subsection.

\begin{thm}
\label{thm:equiv}
The functor $\Phi^0$ is an equivalence of categories.
\end{thm}

\subsection{Proof of Theorem~\ref{thm:equiv}}

As explained in~\S\ref{ss:def-IW}, any ring of coefficients considered above appears in an $\ell$-modular triple $(\K,\O,\L)$ where $\K$ is a finite extension of $\Ql$, $\O$ is its ring of integers, and $\L$ is the residue field of $\O$. Therefore we fix such a triple, and will treat the three cases in parallel.

The starting point of our proof will be the \emph{geometric Casselman--Shalika formula}, first conjectured in~\cite{fgkv} and then proved independently in~\cite{fgv} and~\cite{ngo-polo} (see also Remark~\ref{rmk:Hom-J!-J*-IW}\eqref{it:J!-J*-IW-simple}). We consider the composition
\[
 \chi_{U^+_\scK} : U^+_\scK \xrightarrow{\chi_{\scK}} (\Ga)_{\scK} \to \Ga,
\]
where the second map is the ``residue'' morphism defined by
\[
 \sum_{i \in \Z} f_i z^i \mapsto f_{-1}.
\]
For $\mu \in \bX^+$ we set $S_\mu := U^+_\mathscr{K} \cdot L_\mu$; then there exists a unique function $\chi_\mu : S_\mu \to \Ga$ such that $\chi_\mu(u \cdot L_\mu)=\chi_{U^+_\scK}(u)$ for any $u \in U^+_\mathscr{K}$. The geometric Casselman--Shalika formula states that for $\lambda,\mu \in \bX^\vee_+$ we have
\begin{equation}
\label{eqn:CasselmanShalika}
\mathsf{H}^i_c \bigl( S_\mu, \cJ_{!*}(\lambda,\K)_{|S_\mu} \otimes_\K \chi_\mu^*(\cL_\psi^\K) \bigr) = \begin{cases}
\K & \text{if $\lambda=\mu$ and $i=\langle 2\rho,\lambda\rangle$;} \\
0 & \text{otherwise.}
\end{cases}
\end{equation}

In the following lemma, we denote by $\chi_\mu' : z^{-\varsigma} X_{\mu+\varsigma} \to \Ga$ the unique function such that $\chi_\mu'(z^{-\varsigma} \cdot u \cdot L_{\mu+\varsigma}) = \chi_{I^+}(u)$ for $u \in I_\uu^+$.

\begin{lem}
\label{lem:vanishing-CS}
For $\bk \in \{\K,\O,\L\}$, for any $\lambda,\mu \in \bX^+$ with $\lambda \neq \mu$ we have
\[
\mathsf{H}^{\langle \lambda+\mu,2\rho \rangle}_c \bigl( \Gr^\lambda \cap (z^{-\varsigma} X_{\mu+\varsigma}), (\chi_\mu')^*(\cL^\bk_\psi)_{|\Gr^\lambda \cap (z^{-\varsigma} X_{\mu+\varsigma})}) = 0.
\]
\end{lem}

\begin{proof}
For $\alpha \in \Delta$ and $n \in \Z_{\geq 0}$ we denote by $U_{\alpha,n} \subset \GO$ the image of the morphism $x \mapsto u_\alpha(xz^n)$.
As explained e.g.~in~\cite[Lemme~2.2]{ngo-polo}, the action on $L_{\mu+\varsigma}$ induces an isomorphism
\[
\prod_{\alpha \in \Delta^+} \prod_{j=0}^{\langle \mu+\varsigma, \alpha \rangle -1} U_{\alpha,j} \simto X_{\mu + \varsigma}.
\]
Multiplying by $z^{-\varsigma}$ we deduce that $z^{-\varsigma} X_{\mu+\varsigma} \subset S_\mu$, and moreover that $\chi_\mu'$ is the restriction of $\chi_\mu$ to $z^{-\varsigma} X_{\mu+\varsigma}$. By~\cite[Theorem~3.2]{mv}, we have $\dim(\Gr^\lambda \cap S_\mu)=\langle \lambda+\mu, \rho \rangle$; it follows that $\dim(\Gr^\lambda \cap (z^{-\varsigma} X_{\mu+\varsigma})) \leq \langle \lambda+\mu, \rho \rangle$. If this inequality is strict, then our vanishing claim is obvious (see e.g.~\cite[Theorem~I.8.8]{fk}). Otherwise, each irreducible component of $\Gr^\lambda \cap (z^{-\varsigma} X_{\mu+\varsigma})$ of dimension $\langle \lambda+\mu, \rho \rangle$ is dense (hence open) in an irreducible component of $\Gr^\lambda \cap S_\mu$; therefore to prove the lemma it suffices to prove that
\begin{equation}
\label{eqn:vanishing-CasselmanShalika}
\mathsf{H}^{\langle \lambda+\mu,2\rho \rangle}_c \bigl( \Gr^\lambda \cap S_\mu, \chi_\mu^*(\cL^\bk_\psi)_{|\Gr^\lambda \cap S_\mu}) = 0.
\end{equation}
Finally, we note that since we are considering the top cohomology, the $\O$-module $\mathsf{H}^{\langle \lambda+\mu,2\rho \rangle}_c \bigl( \Gr^\lambda \cap S_\mu, \chi_\mu^*(\cL^\O_\psi)_{|\Gr^\lambda \cap S_\mu})$ is free, and the natural morphisms
\begin{align*}
\K \otimes_\O \mathsf{H}^{\langle \lambda+\mu,2\rho \rangle}_c \bigl( \Gr^\lambda \cap S_\mu, \chi_\mu^*(\cL^\O_\psi)_{|\Gr^\lambda \cap S_\mu}) &\to \mathsf{H}^{\langle \lambda+\mu,2\rho \rangle}_c \bigl( \Gr^\lambda \cap S_\mu, \chi_\mu^*(\cL^\K_\psi)_{|\Gr^\lambda \cap S_\mu}), \\
\L \otimes_\O \mathsf{H}^{\langle \lambda+\mu,2\rho \rangle}_c \bigl( \Gr^\lambda \cap S_\mu, \chi_\mu^*(\cL^\O_\psi)_{|\Gr^\lambda \cap S_\mu}) &\to \mathsf{H}^{\langle \lambda+\mu,2\rho \rangle}_c \bigl( \Gr^\lambda \cap S_\mu, \chi_\mu^*(\cL^\L_\psi)_{|\Gr^\lambda \cap S_\mu})
\end{align*}
are isomorphisms; hence it suffices to prove~\eqref{eqn:vanishing-CasselmanShalika} in case $\bk=\K$.

So, from now on we assume that $\bk=\K$. The geometric Casselman--Shalika formula~\eqref{eqn:CasselmanShalika} implies 
that for any $\cF$ in $\Perv_{\GO}(\Gr,\K)$ we have $\mathsf{H}^i_c(S_\mu, \mathcal{F}_{|S_\mu} \otimes_\K \chi_\mu^*(\cL_\psi^\K))=0$ for $i \neq \langle 2\rho,\mu \rangle$; therefore the morphism $(j_\lambda)_! \underline{\K}_{\Gr^\lambda}[\langle \lambda, 2\rho \rangle] \to \cJ_!(\lambda,\K)$ induces an isomorphism
\[
\mathsf{H}^{\langle \lambda+\mu,2\rho \rangle}_c \Bigl( S_\mu, \bigl((j_\lambda)_! \underline{\K}_{\Gr^\lambda} \bigr)_{|S_\mu} \otimes_\K \chi_\mu^*(\cL_\psi^\K) \Bigr) \simto \mathsf{H}^{\langle \mu,2\rho \rangle}_c \bigl( S_\mu, \cJ_!(\lambda,\K) \otimes_\K \chi_\mu^*(\cL_\psi^\K) \bigr).
\]
Now we have $\cJ_!(\lambda,\K) \cong \cJ_{!*}(\lambda,\K)$ by Remark~\ref{rmk:Hom-J!-J*}\eqref{it:semisimple-J*-J!}; hence the right-hand side vanishes if $\lambda \neq \mu$ by~\eqref{eqn:CasselmanShalika}. On the other hand, the base change theorem shows that the left-hand side identifies with $\mathsf{H}^{\langle \lambda+\mu,2\rho \rangle}_c \bigl( \Gr^\lambda \cap S_\mu, \chi_\mu^*(\cL^\K_\psi)_{|\Gr^\lambda \cap S_\mu})$; we have therefore proved~\eqref{eqn:vanishing-CasselmanShalika} in this case, hence the lemma.
\end{proof}

\begin{prop}
\label{prop:vanishing-Hom-Phi}
For $\bk \in \{\K,\L\}$, for any $\lambda,\mu \in \bX^+$ with $\lambda \neq \mu$ we have
\[
\Hom_{\Perv_{\IW}(\Gr,\bk)} \bigl( \Phi^0(\cJ_!(\lambda,\bk)), \nabla^{\IW}_{\mu+\varsigma}(\bk) \bigr)=0.
\]
\end{prop}

\begin{proof}
First, by exactness of $\Phi$ we see that the morphism $(j_\lambda)_! \underline{\bk}_{\Gr^\lambda}[\langle \lambda, 2\rho \rangle] \to \cJ_!(\lambda,\bk)$ induces an isomorphism
\begin{multline*}
\Hom_{\Db_{\IW}(\Gr,\bk)} \bigl( \Phi((j_\lambda)_! \underline{\bk}_{\Gr^\lambda}[\langle \lambda, 2\rho \rangle]), \nabla^{\IW}_{\mu+\varsigma}(\bk) \bigr) \\
\simto \Hom_{\Perv_{\IW}(\Gr,\bk)} \bigl( \Phi^0(\cJ_!(\lambda,\bk)), \nabla^{\IW}_{\mu+\varsigma}(\bk) \bigr).
\end{multline*}

Let $J$ be the stabilizer of the point $L_\varsigma$ in $I_\uu^+$. Then $\chi_{I^+}^*(\cL_\psi^\bk)$ is trivial on $J$, so that we have a forgetful functor
\[
\Db_{\IW}(\Gr,\bk) \to \Db_J(\Gr,\bk).
\]
By the same considerations as in~\S\ref{ss:categories}, this functor admits a left adjoint, denoted $\lInd_{J}^{(I_\uu^+,\chi_{I^+})}$. Moreover, since $J$ is pro-unipotent the forgetful functor $\Db_J(\Gr,\bk) \to \Db_c(\Gr,\bk)$ is fully faithful. 

If we denote by $\mathsf{F}_{\varsigma}$ the direct image under the automorphism $x \mapsto z^{\varsigma} \cdot x$ of $\Gr$, then from the definition we see that
\begin{equation}
\label{eqn:Phi-Ind-For}
\Phi(\cF) = \lInd_{J}^{(I_\uu^+,\chi_{I^+})} \circ \mathsf{F}_{\varsigma} \circ \For^{\GO}_{z^{-\varsigma} J z^{\varsigma}}(\cF)[-\langle \varsigma,2\rho \rangle]
\end{equation}
for any $\cF$ in $\Db_{\GO}(\Gr,\bk)$. In the setting of the proposition, we deduce an isomorphism
\begin{multline*}
\Hom_{\Db_{\IW}(\Gr,\bk)} \bigl( \Phi((j_\lambda)_! \underline{\bk}_{\Gr^\lambda}[\langle \lambda, 2\rho \rangle]), \nabla^{\IW}_{\mu+\varsigma}(\bk) \bigr) \\
\cong \Hom_{\Db_c(\Gr,\bk)}\bigl( (j_\lambda)_! \underline{\bk}_{\Gr^\lambda}[\langle \lambda, 2\rho \rangle], \mathsf{F}_{\varsigma}^{-1}(\nabla^\IW_{\mu+\varsigma}(\bk)) [\langle \varsigma,2\rho\rangle] \bigr).
\end{multline*}
Now $\mathsf{F}_{\varsigma}^{-1}(\nabla^\IW_{\mu+\varsigma}(\bk))$ identifies with the $*$-pushforward of $(\chi_\mu')^*(\cL^\bk_\psi)[\langle \mu+\varsigma, 2\rho \rangle]$ under the embedding $z^{-\varsigma} X_{\mu+\varsigma} \to \Gr$. Hence, by the base change theorem, the right-hand side identifies with
\[
\mathsf{H}^{\langle \mu+2\varsigma-\lambda,2\rho \rangle}( \Gr^\lambda \cap z^{-\varsigma} X_{\mu+\varsigma}, a^! (\chi_\mu')^*(\cL^\bk_\psi)),
\]
where $a : \Gr^\lambda \cap z^{-\varsigma} X_{\mu+\varsigma} \hookrightarrow z^{-\varsigma} X_{\mu+\varsigma}$ is the embedding.

So, we now need to show that $\mathsf{H}^{\langle \mu+2\varsigma-\lambda,2\rho \rangle}(\Gr^\lambda \cap z^{-\varsigma} X_{\mu+\varsigma}, a^! (\chi_\mu')^*(\cL^\bk_\psi))$ vanishes. If $b$ denotes the unique map $z^{-\varsigma} X_{\mu+\varsigma} \to \mathrm{pt}$, then by Verdier duality we have
\begin{multline*}
 \mathsf{H}^{\langle \mu+2\varsigma-\lambda,2\rho \rangle}( \Gr^\lambda \cap z^{-\varsigma} X_{\mu+\varsigma}, a^! (\chi_\mu')^*(\cL^\bk_\psi))^* = \mathsf{H}^{\langle \mu+2\varsigma-\lambda,2\rho \rangle}( b_* a^! (\chi_\mu')^*(\cL^\bk_\psi))^* \\
 \cong \mathsf{H}^{\langle \lambda-2\varsigma-\mu,2\rho \rangle}( b_! a^* \mathbb{D}_{z^{-\varsigma} X_{\mu+\varsigma}}((\chi_\mu')^*(\cL^\bk_\psi))).
\end{multline*}
Now since $z^{-\varsigma} X_{\mu+\varsigma}$ is smooth of dimension $\langle \mu+\varsigma,2\rho\rangle$ we have an isomorphism $\mathbb{D}_{z^{-\varsigma} X_{\mu+\varsigma}}((\chi_\mu')^*(\cL^\bk_\psi)) \cong (\chi_\mu')^*(\cL^\bk_{-\psi})[2\langle \mu+\varsigma,2\rho\rangle]$, which shows that
\begin{multline*}
 \mathsf{H}^{\langle \mu+2\varsigma-\lambda,2\rho \rangle}( \Gr^\lambda \cap z^{-\varsigma} X_{\mu+\varsigma}, a^! (\chi_\mu')^*(\cL^\bk_\psi))^* \\
 \cong \mathsf{H}_c^{\langle \mu+\lambda,2\rho \rangle}( \Gr^\lambda \cap z^{-\varsigma} X_{\mu+\varsigma}, a^* (\chi_\mu')^*(\cL^\bk_{-\psi})).
\end{multline*}
The right-hand side vanishes by Lemma~\ref{lem:vanishing-CS}, hence so does the left-hand side, which completes the proof.
\end{proof}

We can finally give the proof of Theorem~\ref{thm:equiv}.

\begin{proof}[Proof of Theorem~{\rm \ref{thm:equiv}}]
Let $\tau : G_\scK \to \Gr$ be the projection. Let $\lambda \in \bX^\vee_+$, and denote by $m_{\varsigma,\lambda}$ the restriction of $m^{\Gr}$ to $\tau^{-1}(\overline{\strut \Gr^\varsigma}) \times^{G_\scO} \overline{\Gr^\lambda}$ (see~\S\ref{ss:convolution} for the notation). Then it is well known that:
\begin{itemize}
\item
$m_{\varsigma,\lambda}$ takes values in $\overline{\Gr^{\lambda+\varsigma}} = \overline{\strut X_{\lambda+\varsigma}}$;
\item
its restriction to the preimage of $X_{\lambda+\varsigma}$ is an isomorphism;
\item
this preimage is contained in $\tau^{-1}(X_\varsigma) \times^{\GO} \Gr^\lambda$.
\end{itemize}
These properties imply
that the perverse sheaf $\Phi^0(\cJ_!(\lambda,\bk))$ is supported on $\overline{X_{\lambda+\varsigma}}$, and that its restriction to $X_{\lambda+\varsigma}$ is a perversely shifted local system of rank $1$. The same comments apply to $\Phi^0(\cJ_*(\lambda,\bk))$.
Hence there exist canonical morphisms
\[
f^\bk_\lambda : \Delta^\IW_{\lambda + \varsigma}(\bk) \to \Phi^0(\cJ_!(\lambda,\bk)) \quad \text{and} \quad g_\lambda^\bk : \Phi^0(\cJ_*(\lambda,\bk)) \to \nabla^\IW_{\lambda + \varsigma}(\bk)
\]
whose restrictions to $X_{\lambda+\varsigma}$ are isomorphisms.

We claim that $f_\lambda^\bk$ is an isomorphism.
First we note that all of our constructions are compatible with extension-of-scalars functors in the obvious sense (see in particular~\cite[Proposition~8.1]{mv} for the case of $\cJ_!(\lambda,\bk)$; the case of the Whittaker standard object is much easier since no perverse truncation is involved). If $\bk \in \{\K,\L\}$, by Proposition~\ref{prop:vanishing-Hom-Phi} we know that $\Phi^0(\cJ_!(\lambda,\bk))$ has no quotient of the form $\IC^\IW_{\mu+\varsigma}(\bk)$ with $\mu \neq \lambda$; therefore $f_\lambda^\bk$ is surjective. 
The surjectivity of $f_\lambda^\L$ implies that $f_\lambda^{\mathbb{O}}$ must be surjective also. On the other hand, by Remark~\ref{rmk:Hom-J!-J*-IW}\eqref{it:J!-J*-IW-simple} the object $\Delta^\IW_{\lambda + \varsigma}(\K)$ is simple; hence $f_\lambda^{\mathbb{K}}$ is injective, which implies that $\ker(f_\lambda^{\mathbb{O}})$ is a torsion object. Since this object embeds in the torsion-free object $\Delta^\IW_{\lambda + \varsigma}(\mathbb{O})$, it must be zero. We finally obtain that $f_\lambda^{\mathbb{O}}$ is an isomorphism, so that $f_\lambda^\K$ and $f_\lambda^\L$ are isomorphisms as well.

Once we know that $f_\lambda^\bk$ is an isomorphism, by Verdier duality (see the comments preceding Lemma~\ref{lem:Phi-exact}) we deduce that $g_\lambda^\bk$ is an isomorphism as well. (More precisely, we use the claim about $f_\lambda^\bk$ in the setting where $\psi$ is replaced by $\psi^{-1}$, and the fact that $\mathbb{D}_{\Gr}(\cJ_!(\lambda,\bk)) = \cJ_*(\lambda,\bk)$, see~\cite[Proposition~8.1(c)]{mv}.)

Now we
conclude the proof as follows.
Since $\Phi^0$ is exact, it induces a functor
\[
\Db(\Phi^0) : \Db \Perv_{G_\scO}(\Gr, \bk) \to \Db \Perv_{\IW}(\Gr, \bk).
\]
We will prove that $\Db(\Phi^0)$ is an equivalence, which will imply that $\Phi^0$ is an equivalence as well, hence will conclude the proof.
It is not difficult to see that the category $\Db \Perv_{G_\scO}(\Gr, \bk)$, resp.~$\Db \Perv_{\IW}(\Gr, \bk)$, is generated as a triangulated category by the objects $\{\cJ_!(\lambda,\bk) : \lambda \in \bX^\vee_+\}$, resp.~by the objects $\{\Delta^\IW_{\lambda+\varsigma}(\bk) : \lambda \in \bX^\vee_+\}$, as well as by the objects $\{\cJ_*(\lambda,\bk) : \lambda \in \bX^\vee_+\}$, resp.~by the objects $\{\nabla^\IW_{\lambda+\varsigma}(\bk) : \lambda \in \bX^\vee_+\}$. Hence to conclude it suffices to prove that for any $\lambda,\mu \in \bX^\vee_+$ and any $n \in \Z$ the functor $\Phi^0$ induces an isomorphism
\[
 \Ext^n_{\Perv_{G_\scO}(\Gr, \bk)}(\cJ_!(\lambda,\bk), \cJ_*(\mu,\bk)) \xrightarrow{\sim} \Ext^n_{\Perv_{\IW}(\Gr, \bk)}(\Delta^\IW_{\lambda+\varsigma}(\bk), \nabla^\IW_{\mu+\varsigma}(\bk)).
\]
However, this is clear from Remark~\ref{rmk:Hom-J!-J*}\eqref{it:Hom-J!-J*} and Remark~\ref{rmk:Hom-J!-J*-IW}\eqref{it:Hom-J!-J*-IW}.
\end{proof}

\begin{rmk}
\begin{enumerate}
\item
One can explicitly describe the inverse to $\Phi^0$, as follows. In view of~\eqref{eqn:Phi-Ind-For}, the functor
\[
\Psi :=\rInd_{z^{-\varsigma} J z^{\varsigma}}^{\GO} \circ \mathsf{F}_\varsigma^{-1} \circ \For^{(I_\uu^+, \chi_{I^+})}_J [\langle \varsigma,2\rho \rangle] : \Db_\IW(\Gr,\bk) \to \Db_{\GO}(\Gr,\bk)
\]
is right adjoint to $\Phi$. Since $\Phi$ is exact, $\Psi$ is left exact, and the functor $\Psi^0:=\pH^0 \circ \Psi_{|\Perv_\IW(\Gr,\bk)}$ is right adjoint to $\Phi^0$. Since $\Phi^0$ is an equivalence, $\Psi^0$ must be its inverse.
\item
 From the point of view suggested by the Finkelberg--Mirkovi\'c conjecture (see~\S\ref{ss:intro-fm}), the isomorphisms $f^\bk_\lambda$ and $g^\bk_\lambda$ are geometric analogues of the isomorphism stated in~\cite[Proposition~II.3.19]{jantzen}.
\end{enumerate}
\end{rmk}

\section{Applications}
\label{sec:applications}

We continue with the assumptions of Sections~\ref{sec:cont-sheaves}--\ref{sec:spherical-IW}; but from now on (except in Remark~\ref{rmk-tilting-ring})
for simplicity we assume that $\bk$ is a field.

\subsection{Some perverse sheaves associated with regular \texorpdfstring{$W$}{W}-orbits in \texorpdfstring{$\bX^\vee$}{X}}
\label{ss:soergel}


Consider the flag variety $\mathscr{B}=G/B^-$, and let $U^-$ be the unipotent radical of $B^-$.
Recall that the category $\Perv_{U^-}(\mathscr{B},\bk)$ of $U^-$-equivariant perverse sheaves on $\mathscr{B}$ has a natural structure of highest weight category, see~\cite{bgs}. Moreover, the projective cover $\mathcal{P}_e$ of the skyscraper sheaf 
at the point $B^-/B^-$ is also an injective and a tilting object; see e.g.~\cite{bezr} for details and references.

For any $\lambda \in \bX^\vee_{++}$, in the notation of~\S\ref{ss:Gr-Fl} we have $P_\lambda=B^-$, so that the map $p_\lambda$ defined there has codomain $\mathscr{B}$.
We set
\[
\cP_\lambda:=(p_\lambda)^*(\cP_e) [\dim(\Gr^\lambda) - \dim(\mathscr{B})].
\]
Then $\cP_\lambda$ is a perverse sheaf on $\Gr^\lambda$, and it is 
$I^-_\uu$-equivariant.
We will consider the objects
\[
\Pi_\lambda^! := (j_{\lambda})_! \cP_\lambda, \quad \Pi_\lambda^* := (j_{\lambda})_* \cP_\lambda.
\]

\begin{lem}
\label{lem:Pi-!*-perverse}
 The objects $\Pi_\lambda^!$ and $\Pi_\lambda^*$ are $I^-_\uu$-equivariant perverse sheaves on $\Gr$.
\end{lem}

\begin{proof}
 As recalled above,
 $\cP_e$ 
 admits both a standard filtration and a costandard filtration. It follows that $\Pi_\lambda^!$, resp.~$\Pi_\lambda^*$, admits a filtration (in the sense of triangulated categories) with subquotients of the form $\Delta^{\Gr}_{v(\lambda)}$, resp.~$\nabla^{\Gr}_{v(\lambda)}$, for $v \in \Wf$. Since these objects are perverse sheaves, it follows that $\Pi_\lambda^!$ and $\Pi_\lambda^*$ are perverse. The fact that these perverse sheaves are $I^-_\uu$-equivariant readily follows from the fact that $\cP_\lambda$ is $I^-_\uu$-equivariant.
\end{proof}


\begin{lem}
\label{lem:Pi-varsigma}
There exists a canonical isomorphism $\Pi_\varsigma^! \cong \Pi_\varsigma^*$.
\end{lem}

\begin{proof}
This claim is proved in the $\mathscr{D}$-modules setting in~\cite[Proposition~15.2]{fg}. The arguments apply verbatim in the present context.
\end{proof}

In view of this lemma, the object $\Pi_\varsigma^! = \Pi_\varsigma^*$ will be denoted $\Pi_\varsigma$.

Recall now that we have the ``negative'' Iwahori subgroup $I^-$ (associated with the negative Borel $B^-$), but also the ``positive'' Iwahori subgroup $I^+$ (associated with the positive Borel $B^+$) which was used to define the Iwahori--Whittaker category.
Let $I_\circ$ be the kernel of the morphism $\GO \to G$. Then $I_\circ = I^-_\uu \cap I_\uu^+$, and the morphism $\chi_{I^+}$ is trivial on $I_\circ$. It follows that there exists a natural forgetful functor
\[
 \For^\IW_{I_\circ} : \Db_{\IW}(\Gr,\bk) \to \Db_{I_\circ}(\Gr,\bk).
\]
We also have a forgetful functor
\[
 \For^{I^-_\uu}_{I_\circ} : \Db_{I^-_\uu}(\Gr,\bk) \to \Db_{I_\circ}(\Gr,\bk)
\]
which admits both a left and a right adjoint, denoted $\lInd_{I_\circ}^{I^-_\uu}$ and $\rInd_{I_\circ}^{I^-_\uu}$ respectively, see~\S\ref{ss:categories}. 
We set
\begin{gather*}
 \Av_{I^-_\uu,*}:= \rInd_{I_\circ}^{I^-_\uu} \circ \For^\IW_{I_\circ} : \Db_{\IW}(\Gr,\bk) \to \Db_{I_\uu^-}(\Gr,\bk); \\
 \Av_{I^-_\uu,!}:= \lInd_{I_\circ}^{I^-_\uu} \circ \For^\IW_{I_\circ} : \Db_{\IW}(\Gr,\bk) \to \Db_{I^-_\uu}(\Gr,\bk).
\end{gather*}


\begin{lem}
\label{lem:Pi-averaging}
For any $\lambda \in \bX^\vee_{++}$ we have
\begin{equation*}
\Pi_\lambda^! \cong \Av_{I^-_\uu,!} \bigl( \Delta_\lambda^\IW (\bk) \bigr)[-\dim U^-], \quad \Pi_\lambda^* \cong \Av_{I^-_\uu,*} \bigl( \nabla_\lambda^\IW (\bk) \bigr)[\dim U^-].
\end{equation*}
\end{lem}

\begin{proof}
 Consider the constructible equivariant derived categories
 \[
  \Db_{U^-}(\mathscr{B},\bk) \quad \text{and} \quad \Db_{(U^+, \chi^*(\cL^\bk_\psi))}(\mathscr{B},\bk)
 \]
of sheaves on $\mathscr{B}$ which are $U^-$-equivariant and $(U^+, \chi^*(\cL^\bk_\psi))$-equivariant respectively. These categories are related by functors
 \[
  \Av_{U^-,*}, \, \Av_{U^-,!} : \Db_{(U^+, \chi^*(\cL^\bk_\psi))}(\mathscr{B},\bk) \to \Db_{U^-}(\mathscr{B},\bk).
 \]
Moreover, if $\Delta^{(U^+, \chi)}_e$ denotes the $!$-extension of the shift by $\dim U^+$ of the unique simple $(U^+, \chi^*(\cL^\bk_\psi))$-equivariant local system on the orbit $U^+ B^-/B^- \subset \mathscr{B}$ (which also coincides with the $*$-extension of this local system), then it is well known that we have isomorphisms
\begin{equation}
\label{eqn:Pe-Av}
 \Av_{U^-,!} \left( \Delta^{(U^+, \chi)}_e \right) [-\dim U^-] \cong \cP_e \cong \Av_{U^-,*} \left( \Delta^{(U^+, \chi)}_e \right) [\dim U^-],
\end{equation}
see~\cite[\S 4.6]{by} or~\cite[Lemma~5.18]{modrap1}.

Now, the functors $\Av_{I^-_\uu,*}$ and $\Av_{I^-_\uu,!}$ have versions for the variety $\Gr^\lambda$, which we will denote similarly. Clearly we have isomorphisms of functors
\begin{equation}
\label{eqn:Av-jlambda}
 \Av_{I^-_\uu,*} \circ (j_\lambda)_* \cong (j_\lambda)_* \circ \Av_{I^-_\uu,*}, \quad \Av_{I^-_\uu,!} \circ (j_\lambda)_! \cong (j_\lambda)_! \circ \Av_{I^-_\uu,!}.
\end{equation}
Moreover, the map $p_\lambda$ induces a morphism $I^-_\uu \times^{I_\circ} \Gr^\lambda \to U^- \times \mathscr{B}$ compatible with the action maps in the obvious way. Using the base change theorem (and the fact that $p_\lambda$ is smooth), we deduce isomorphisms of functors
\begin{equation}
\label{eqn:Av_plambda}
 \Av_{I^-_\uu,*} \circ (p_\lambda)^* \cong (p_\lambda)^* \circ \Av_{U^-,*}, \quad \Av_{I^-_\uu,!} \circ (p_\lambda)^* \cong (p_\lambda)^* \circ \Av_{U^-,!}.
\end{equation}
Since
\begin{multline*}
\Delta^{\IW}_\lambda(\bk) = (j_\lambda)_! (p_\lambda)^* \Delta^{(U^+, \chi)}_e[\dim \Gr^\lambda - \dim \mathscr{B}] \quad \text{and} \\
\nabla^{\IW}_\lambda(\bk) = (j_\lambda)_* (p_\lambda)^* \Delta^{(U^+, \chi)}_e[\dim \Gr^\lambda - \dim \mathscr{B}]
\end{multline*}
the isomorphisms of the lemma finally follow from~\eqref{eqn:Av-jlambda},~\eqref{eqn:Av_plambda} and~\eqref{eqn:Pe-Av}.
\end{proof}

%
%

The following proposition is the main result of this subsection.

\begin{prop}
\label{prop:Lusztig-formula}
For any $\lambda \in \bX^\vee_{+}$, we have isomorphisms
\[
\Pi_\varsigma \star^{\GO} \cJ_!(\lambda,\bk) \cong \Pi^!_{\lambda + \varsigma}, \qquad \Pi_\varsigma \star^{\GO} \cJ_*(\lambda,\bk) \cong \Pi^*_{\lambda + \varsigma}.
\]
\end{prop}

\begin{proof}
The first isomorphism is obtained by applying the functor $\Av_{I^-_\uu,!}[-\dim U^-]$ to the isomorphism
\[
 \Delta^{\IW}_\varsigma \star^{\GO} \cJ_!(\lambda,\bk) \cong \Delta^{\IW}_{\lambda+\varsigma}(\bk)
\]
(see the proof of Theorem~\ref{thm:equiv})
and then using Lemma~\ref{lem:Pi-averaging} and the fact that $\Av_{I^-_\uu,!}$ commutes with the functor $(-) \star^{\GO} \cF$ for any $\cF$ in $\Db_{\GO}(\Gr,\bk)$. The proof of the second isomorphism is similar, using $\Av_{I^-_\uu,*}$ instead of $\Av_{I^-_\uu,!}$.
\end{proof}

\begin{rmk}
 Consider the restrictions
 \[
  \Av_{I^-_\uu,!}^0, \Av_{I^-_\uu,*}^0 : \Perv_{\IW}(\Gr,\bk) \to \Db_{I^-_\uu}(\Gr,\bk)
 \]
of $\Av_{I^-_\uu,!}$ and $\Av_{I^-_\uu,*}$ to the heart of the perverse t-structure. Then there exists an isomorphism of functors
\begin{equation*}
\Av_{I^-_\uu,!}^0[-\dim U^-] \xrightarrow{\sim} \Av_{I^-_\uu,*}^0[\dim U^-],
\end{equation*}
and moreover these functors take values in $\Perv_{I^-_\uu}(\Gr,\bk)$ and send tilting perverse sheaves to tilting perverse sheaves. (Here the highest weight structure on $\Perv_{I^-_\uu}(\Gr,\bk)$ is the standard one, as considered e.g.~in~\cite[\S 3.3]{bgs}.) In fact, as in the proof of Proposition~\ref{prop:Lusztig-formula}, for any $\cF$ in $\Perv_{\GO}(\Gr,\bk)$ we have
\[
 \bigl( \Av_{I^-_\uu,!}^0[-\dim U^-] \bigr) \circ \Phi^0(\cF) \cong \Pi_{\varsigma} \star^{\GO} \cF \cong \bigl( \Av_{I^-_\uu,*}^0[\dim U^-] \bigr) \circ \Phi^0(\cF),
\]
and then the isomorphism follows from the fact that $\Phi^0$ is an equivalence of categories, see Theorem~\ref{thm:equiv}. Once this fact is established, it follows from Lemma~\ref{lem:Pi-averaging} that this functor sends standard perverse sheaves, resp.~costandard perverse sheaves, to perverse sheaves admitting a standard filtration, resp.~a costandard filtration (see the proof of Lemma~\ref{lem:Pi-!*-perverse}); the other claims follow.
\end{rmk}

\subsection{Interpretation in terms of the Weyl character formula}
\label{ss:geometric-Weyl}

The isomorphisms in Proposition~\ref{prop:Lusztig-formula} can be considered a geometric version of the Weyl character formula as stated by Lusztig in~\cite[(6.3)]{lusztig}, in the following way. 
Let
\[
\mathscr{Z} : \Perv_{G_\scO}(\Gr, \bk) \to \Perv_{I^-}(\Fl, \bk)
\]
be the ``central'' functor
constructed (in terms of nearby cycles) in~\cite{gaitsgory}. 

\begin{lem}
\label{lem:induction-Z}
 There exists a canonical isomorphism of functors
 \[
  \rInd_{I^-}^{\GO} \circ \mathscr{Z} \cong \pi^*.
 \]
\end{lem}

\begin{proof}
 By definition, the functor $\rInd_{I^-}^{\GO}$ is given by convolution with $\underline{\bk}_{\GO/I^-}$ on the left. Since $\mathscr{Z}(\cF)$ is central for any $\cF$ in  $\Perv_{G_\scO}(\Gr, \bk)$ (see~\cite[Theorem~1(b)]{gaitsgory}), $\rInd_{I^-}^{\GO} \circ \mathscr{Z}$ is therefore the composition of $\mathscr{Z}$ with convolution on the \emph{right} with $\underline{\bk}_{\GO/I^-}$, which itself identifies with the functor $\pi^* \pi_*$. The claim follows, since $\pi_* \circ \mathscr{Z} \cong \id$ by~\cite[Theorem~1(d)]{gaitsgory}.
%
\end{proof}

Using this lemma we obtain the following reformulation of Proposition~\ref{prop:Lusztig-formula}.

\begin{prop}
\label{prop:Lusztig-formula-2}
For any $\lambda \in \bX^{\vee}_+$ there exist canonical isomorphisms
\begin{align*}
\bigl( \pi^* \Pi_\varsigma [\dim \mathscr{B}] \bigr) \star^{I^-} \mathscr{Z}(\cJ_!(\lambda, \bk)) &\cong \bigl( \pi^* \Pi^!_{\lambda+\varsigma} [\dim \mathscr{B}] \bigr), \\
\bigl( \pi^* \Pi_\varsigma [\dim \mathscr{B}] \bigr) \star^{I^-} \mathscr{Z}(\cJ_*(\lambda,\bk)) &\cong \bigl( \pi^* \Pi^*_{\lambda+\varsigma} [\dim \mathscr{B}] \bigr).
\end{align*}
\end{prop}

\begin{proof}
 By Lemma~\ref{lem:convolution-induction} we have
 \[
  \bigl( \pi^* \Pi_\varsigma [\dim \mathscr{B}] \bigr) \star^{I^-} \mathscr{Z}(\cJ_!(\lambda,\bk)) \cong \Pi_{\varsigma}[\dim \mathscr{B}] \star^{\GO} \rInd_{I^-}^{\GO} \bigl( \mathscr{Z}(\cJ_!(\lambda,\bk)) \bigr).
 \]
Using Lemma~\ref{lem:induction-Z}, we deduce an isomorphism
\[
 \bigl( \pi^* \Pi_\varsigma [\dim \mathscr{B}] \bigr) \star^{I^-} \mathscr{Z}(\cJ_!(\lambda,\bk)) \cong \Pi_{\varsigma}[\dim \mathscr{B}] \star^{\GO} \pi^*(\cJ_!(\lambda,\bk)).
\]
Now the right-hand side is clearly isomorphic to $\pi^* \bigl( \Pi_{\varsigma}[\dim \mathscr{B}] \star^{\GO} \cJ_!(\lambda,\bk) \bigr)$, and then the first isomorphism of the proposition follows from Proposition~\ref{prop:Lusztig-formula}. The proof of the second isomorphism is similar.
\end{proof}

The Grothendieck group of the category $\Perv_{I^-_\uu}(\Fl,\bk)$, resp.~$\Perv_{G_\scO}(\Gr,\bk)$, identifies naturally with the (integral) group ring $\Z[W]$ of $W$, resp.~with its center $\Z[X^\vee]^{\Wf}$, and under this isomorphism the right convolution with objects of the form $\mathscr{Z}(-)$ corresponds to the natural multiplication map, see~\cite[\S 0.1]{gaitsgory}. (See also~\cite{ab} for this point of view.) Under these identifications, the isomorphisms of Proposition~\ref{prop:Lusztig-formula-2} are categorical incarnations of the identity~\cite[(6.3)]{lusztig}.

\subsection{Tilting objects in the Satake category}
\label{ss:JMW-conj}


Recall the notion of parity complexes in $\Db_{\IW}(\Gr, \bk)$ considered in Remark~\ref{rmk:parity}.
In certain proofs of this subsection we will also consider the $(I^+_\uu, \chi_{I^+}^*(\cL^\bk_\psi))$-equi\-variant constructible derived category of $\Fl$, which we will denote $\Db_{\IW}(\Fl,\bk)$. 
Of course, we can also consider parity complexes in this category, as well as in the $I^-$-equivariant derived category $\Db_{I^-}(\Fl,\bk)$, or in the $G_\scO$-constructible derived category $\Db_{(G_\scO)}(\Gr, \bk)$. (Note that, by definition, an object of $\Db_{I^-}(\Fl,\bk)$ is parity iff its image in the $I^-$-constructible derived category $\Db_{(I^-)}(\Fl,\bk)$ is parity.) In particular, for any $\lambda \in \bX^\vee_+$, we denote by $\cE_\lambda$ the unique indecomposable parity complex in the category $\Db_{(G_\scO)}(\Gr, \bk)$ supported on $\overline{\Gr^\lambda}$ and whose restriction to $\Gr^\lambda$ is $\underline{\bk}_{\Gr^\lambda}[\dim \Gr^\lambda]$ (see~\cite[Theorem~2.12 and \S 4.1]{jmw}). 

Since $\mathsf{H}^\bullet_{\GO}(\mathrm{pt};\bk)$ might not be concentrated in even degrees, in general the theory of~\cite{jmw} does \emph{not} apply in $\Db_{G_\scO}(\Gr,\bk)$. This difficulty will be remedied by the following lemma.

\begin{lem}
\label{lem:parity-equiv-GO}
Any parity object $\cE$ in $\Db_{(\GO)}(\Gr, \bk)$ is a direct summand of a parity object $\cE'$ which belongs to the essential image of the functor $\For_{\GO} : \Db_{\GO}(\Gr, \bk) \to \Db_{(\GO)}(\Gr, \bk)$.
\end{lem}

\begin{proof}
Of course we can assume that $\cE=\cE_\lambda$ for some $\lambda \in \bX^\vee_+$.
Recall that the forgetful functor $\For_{I^-}$ sends indecomposable parity objects to indecomposable parity objects (see~\cite[Lemma~2.4]{mr}). In view of the classification of such objects in the $I^-$-equivariant and $I^-$-constructible derived categories, this means that any $I^-$-constructible parity complex on $\Gr$ belongs to the essential image of $\For_{I^-}$. In particular, there exists a parity complex $\cF$ in $\Db_{I^-}(\Gr,\bk)$ such that $\cE_\lambda \cong \For_{I^-}(\cF)$. Now we set $\cE':=\For_{\GO}(\rInd_{I^-}^{\GO}(\cF))$. Then $\cE'$ is parity as a convolution of parity complexes, see~\cite[Theorem~4.8]{jmw}. And since this object is supported on $\overline{\Gr^\lambda}$ and has nonzero restriction to $\Gr^\lambda$, it must admit a cohomological shift of $\cE_\lambda$ as a direct summand.
\end{proof}

\begin{rmk}
If $\mathrm{char}(\bk)$ is not a torsion prime for $G$, then $\mathsf{H}^\bullet_{\GO}(\mathrm{pt};\bk)$ \emph{is} concentrated in even degrees; see~\cite[\S 2.6]{jmw}. In this case the parity objects in $\Db_{G_\scO}(\Gr,\bk)$ are well behaved, and one can easily show that in fact any parity object in $\Db_{(\GO)}(\Gr, \bk)$ belongs to the essential image of the functor $\For_{\GO}$.
\end{rmk}


Recall that the forgetful functor $\mathrm{For}_{G_\scO} : \Db_{G_\scO}(\Gr,\bk) \to \Db_{(G_\scO)}(\Gr,\bk)$ restricts to an equivalence between $\GO$-equivariant and $\GO$-constructible perverse sheaves, see~\cite[Proposition~2.1]{mv} (or~\cite[Proposition~1.10.8]{br}). Therefore for any $\cF$ in $\Db_{(G_\scO)}(\Gr, \bk)$ and any $n \in \Z$, the perverse sheaf $\pH^n(\cF)$ is $G_\scO$-equivariant. The main result of this section is the following.

\begin{thm}
\label{thm:parity-tilting}
For any $n \in \mathbb{Z}$ and $\lambda \in \bX^\vee_+$, the $G_\scO$-equivariant perverse sheaf $\pH^n(\cE_\lambda)$ is tilting in the highest weight category $\Perv_{G_\scO}(\Gr, \bk)$. In particular, the indecomposable tilting object associated with $\lambda$ is a direct summand of $\pH^0(\cE_\lambda)$.
\end{thm}

\begin{rmk}\phantomsection
\label{rmk:parities}
\begin{enumerate}
\item
Theorem~\ref{thm:parity-tilting} was stated as a conjecture (in the case $n=0$) in~\cite{jmw2}.
\item
If $\mathrm{char}(\bk)$ is good for $G$, it is known that the objects $\cE_\lambda$ are actually perverse, see~\cite{mr}. (This property was proved earlier in~\cite{jmw2} under slightly stronger assumptions; it is known to be false in bad characteristic, see~\cite{jmw2}.) Hence in Theorem~\ref{thm:parity-tilting} we in fact know that the indecomposable tilting object associated with $\lambda$ is $\pH^0(\cE_\lambda)=\cE_\lambda$.
In general, it seems natural to expect that $\pH^0(\cE_\lambda)$ is indecomposable; but we do not have a proof of (or strong evidence for) this fact.
\item
Since our proof of Theorem~\ref{thm:parity-tilting} relies on Theorem~\ref{thm:equiv}, we have stated it with the same assumptions on $G$. However, once it is known in this generality standard arguments allow to extend its validity to general connected reductive groups; see e.g.~\cite[\S 3.4]{jmw2} for details. Similarly, the analogous claim in the setting of the classical topology on the complex counterpart of $\Gr$ follows from its \'etale version using the general considerations of~\cite[\S 6.1]{bbd}.
\end{enumerate}
\end{rmk}

%

The proof of Theorem~\ref{thm:parity-tilting} requires a few preliminaries. We start with the following observation, which will be crucial for us.

\begin{prop}
\label{prop:parity-IW-Gr}
The parity objects in $\Db_{\IW}(\Gr, \bk)$ are exactly the direct sums of cohomological shifts of tilting perverse sheaves.
\end{prop}

\begin{proof}
As already noticed in the proof of Corollary~\ref{cor:Perv-IW},
the strata $X_\lambda \subset \Gr$ supporting $(I^+_\uu, \chi_{I^+}^*(\cL^\bk_\psi))$-equivariant local systems (i.e.~those with $\lambda \in \bX^\vee_{++}$) have dimensions of constant parity on each connected component of $\Gr$. Therefore, the tilting objects in the highest weight category $\Perv_\IW(\Gr, \bk)$ are also parity. By unicity, they must then coincide with the ``parity sheaves'' (or, in another terminology, normalized indecomposable parity complexes) of~\cite[Definition~2.14]{jmw}. The claim follows, since any parity complex is a direct sum of cohomological shifts of such objects.
\end{proof}

Next we observe that the parity property is preserved under convolution, in the following sense.

\begin{lem}
\label{lem:pi-parity}
 If $\cF \in \Db_{\IW}(\Fl,\bk)$ and $\cG \in \Db_{I^-}(\Fl,\bk)$ are parity complexes, then $\cF \star^{I^-} \cG \in \Db_{\IW}(\Fl,\bk)$ is a parity complex.
\end{lem}

\begin{proof}
 In view of the description of parity complexes in~\cite[\S 4.1]{jmw}, the claim follows from standard arguments going back at least to~\cite{springer}. In fact it suffices to treat the case $\cG= \underline{\bk}_{\overline{\Fl_w}}$ when $\ell(w) \in \{0,1\}$, which can be done ``by hand'' as in~\cite{springer}.
\end{proof}

\begin{lem}
\label{lem:convolution-parity}
If $\cF \in \Db_{\IW}(\Gr,\bk)$ is parity and $\cG \in \Db_{G_\scO}(\Gr,\bk)$ is such that $\mathrm{For}_{G_{\scO}}(\cG)$ is parity, then $\cF \star^{G_\scO} \cG \in \Db_{\IW}(\Gr,\bk)$ is parity.
\end{lem}

\begin{proof}
The natural projection $\pi : \Fl \to \Gr$ (a smooth and projective morphism) is $I_\uu^+$-equivariant, hence induces functors
\[
 \pi^* : \Db_{\IW}(\Gr, \bk) \to \Db_{\IW}(\Fl, \bk), \quad \pi_* : \Db_{\IW}(\Fl, \bk) \to \Db_{\IW}(\Gr, \bk).
\]
The projection formula shows that $\cF$ is a direct summand in $\pi_*(\pi^* \cF)$, and by Lemma~\ref{lem:convolution-For} we have
\[
\pi_*(\pi^* \cF) \star^{G_\scO} \cG \cong \pi^*(\cF) \star^{I^-} \For^{\GO}_{I^-} (\cG).
\]
Hence to conclude it suffices to prove that $\pi^*(\cF) \star^{I^-} \For^{\GO}_{I^-}(\cG)$ is parity.
However we have
\[
\pi^* \bigl( \pi^*(\cF) \star^{I^-} \For^{\GO}_{I^-}(\cG) \bigr) \cong \pi^*(\cF) \star^{I^-} \pi^*(\For^{\GO}_{I^-}(\cG)).
\]
Since $\pi^* \cF$ and $\pi^*(\For^{\GO}_{I^-}(\cG))$ are parity (because $\pi$ is smooth),
Lemma~\ref{lem:pi-parity} implies that $\pi^* \bigl( \pi^*(\cF) \star^{I^-} \For^{\GO}_{I^-}(\cG) \bigr)$ is parity. We deduce that $\pi^*(\cF) \star^{I^-} \For^{\GO}_{I^-}(\cG)$ is parity, as expected.
\end{proof}

\begin{cor}
\label{cor:Phi-parity}
Let $\cE$ be in $\Db_{G_\scO}(\Gr, \bk)$, and assume that $\mathrm{For}_{G_\scO}(\cE)$ is parity. Then $\Phi(\cE)$ is parity in $\Db_{\IW}(\Gr, \bk)$. In particular, $\Phi^0(\pH^n(\cE))$ is a tilting perverse sheaf for any $n \in \Z$.
\end{cor}

\begin{proof}
Since $\Delta^\IW_{\varsigma}(\bk)$ is parity (see~\eqref{eqn:Delta-nabla-sigma}), the first claim follows from Lemma~\ref{lem:convolution-parity}. The second claim follows from the facts that $\Phi$ is t-exact and that the perverse cohomology objects of parity objects in $\Db_{\IW}(\Gr, \bk)$ are tilting perverse sheaves, see Proposition~\ref{prop:parity-IW-Gr}.
\end{proof}

We can finally give the proof of Theorem~\ref{thm:parity-tilting}.

\begin{proof}[Proof of Theorem~{\rm \ref{thm:parity-tilting}}]
Since $\Phi^0$ is an equivalence of highest weight categories, to prove the first claim it suffices to prove that $\Phi^0(\pH^n(\cE_\lambda))$ is tilting in the highest weight category $\Perv_\IW(\Gr,\bk)$. This follows from Lemma~\ref{lem:parity-equiv-GO} and Corollary~\ref{cor:Phi-parity}, since a direct summand of a tilting perverse sheaf is tilting.
%
The second claim follows, since $\pH^0(\cE_\lambda)$ is supported on $\overline{\Gr^\lambda}$, and has nonzero restriction to $\Gr^\lambda$.
\end{proof}

\subsection{Convolution and restriction of tilting objects}

In this subsection we will consider the affine Grassmannian for several reductive groups, so we write $\Gr_G$ instead of $\Gr$.
For $P \subset G$ a parabolic subgroup containing $B^+$, with Levi subgroup containing $T$ denoted $L$, we denote by
\[
\mathsf{R}^G_L : \Db_{G_\scO}(\Gr_G,\bk) \to \Db_{L_\scO}(\Gr_L,\bk)
\]
the ``renormalized'' hyperbolic localization functor defined as follows. The connected components of $\Gr_L$ are in a canonical bijection with $\bX^\vee / \Z \Delta^\vee_L$, where $\Delta^\vee_L$ is the coroot system of $(L,T)$; the connected component associated with $c$ will be denoted $\Gr_L^c$. We denote by $U_P^+$ the unipotent radical of $P$. Then for $c \in \bX^\vee / \Z \Delta^\vee_L$ we consider the subvariety
\[
S_c := (U_P^+)_{\scK} \cdot \Gr_L^c
\]
of $\Gr_G$. We denote the natural maps by
\[
\xymatrix{
\Gr_G & S_c \ar[l]_-{s_c} \ar[r]^-{\sigma_c} & \Gr_L^c
}.
\]
Then if $\Delta_L \subset \Delta$ is the root system of $(L,T)$, the functor $\mathsf{R}^G_L$ is defined as
\[
\bigoplus_{c \in \bX^\vee / \Z \Delta^\vee_L} (\sigma_c)_! (s_c)^* [-\langle \sum_{\alpha \in \Delta^+ \smallsetminus \Delta_L} \alpha ,c\rangle].
\]
By work of Be{\u\i}linson--Drinfeld~\cite{bd} this functor is known to be exact for the perverse t-structures; see~\cite[Lem\-ma~1.15.1]{br} for a more detailed proof.

As a consequence of Theorem~\ref{thm:parity-tilting} (and its proof) we obtain the following result, which is a geometric version of a celebrated result on tilting modules due in full generality to Mathieu~\cite{mathieu}. (See~\cite[\S 1.1]{jmw2} for more historical remarks and references on this result). In fact (as noted in the introduction), reasoning in the opposite direction, combining this result with the geometric Satake equivalence, our work can also be considered as providing a new proof of this representation-theoretic result.

\begin{thm}
\label{thm:conv-tilting}
\begin{enumerate}
\item
\label{it:conv-tilting-1}
If $\cF, \cG$ are tilting objects in $\Perv_{G_\scO}(\Gr, \bk)$, then so is $\cF \star^{G_\scO} \cG$.
\item
\label{it:conv-tilting-2}
If $\cF$ is a tilting object in $\Perv_{G_\scO}(\Gr_G,\bk)$, then $\mathsf{R}^G_L(\cF)$ is a tilting object in $\Perv_{L_\scO}(\Gr_L,\bk)$.
\end{enumerate}
\end{thm}

\begin{proof}
\eqref{it:conv-tilting-1}
In view of Theorem~\ref{thm:parity-tilting}, it suffices to show that if $\cF,\cG$ are parity objects in $\Db_{(G_\scO)}(\Gr, \bk)$, then $\pH^0(\cF) \star^{G_\scO} \pH^0(\cG)$ is a tilting perverse sheaf. 
In view of Lemma~\ref{lem:parity-equiv-GO}, it suffices to consider the case when $\cG=\mathrm{For}_{G_\scO}(\cG')$ for some $\cG'$ in $\Db_{G_\scO}(\Gr, \bk)$.
Then by exactness of convolution with $\GO$-equivariant perverse sheaves (see Lemma~\ref{lem:convolution-exact}) we have
\[
\pH^0(\cF) \star^{G_\scO} \pH^0(\cG) = \pH^0(\pH^0(\cF) \star^{G_\scO} \cG').
\]
Hence, using the t-exact functor $\Phi$ of~\S\ref{ss:equivalence}, we see that to conclude it suffices to prove that
\[
\Phi(\pH^0(\pH^0(\cF) \star^{G_\scO} \cG')) \cong \pH^0(\Phi(\pH^0(\cF) \star^{G_\scO} \cG')) \cong \pH^0(\Phi(\pH^0(\cF)) \star^{G_\scO} \cG')
\]
(where the second identification uses the canonical isomorphism
$\Phi(\mathcal{M} \star^{G_\scO} \mathcal{N}) \cong \Phi(\mathcal{M}) \star^{G_\scO} \mathcal{N}$
for $\mathcal{M}, \mathcal{N}$ in $\Db_{G_\scO}(\Gr, \bk)$)
is a tilting object in $\Perv_{\IW}(\Gr, \bk)$. However $\Phi(\pH^0(\cF))$ is a tilting perverse sheaf by Theorem~\ref{thm:parity-tilting}, hence it is also parity by Proposition~\ref{prop:parity-IW-Gr}. By Lemma~\ref{lem:convolution-parity}, it follows that $\Phi(\pH^0(\cF)) \star^{G_\scO} \cG'$ is parity, hence that its perverse cohomology objects are tilting (see again Proposition~\ref{prop:parity-IW-Gr}), which finishes the proof.

\eqref{it:conv-tilting-2}
As in~\eqref{it:conv-tilting-1}, it suffices to prove that if $\cF$ is a parity object in $\Db_{(G_\scO)}(\Gr, \bk)$, then $\mathsf{R}^G_L(\pH^0(\cF))$ is a tilting perverse sheaf. However, since $\mathsf{R}^G_L$ is t-exact we have
\[
\mathsf{R}^G_L(\pH^0(\cF)) \cong \pH^0(\mathsf{R}^G_L(\cF)).
\]
By~\cite[Theorem~1.6]{jmw2}, $\mathsf{R}^G_L(\cF)$ is a parity complex. Then the claim follows from Theorem~\ref{thm:parity-tilting}.
\end{proof}

%
%

\begin{rmk}
\label{rmk-tilting-ring}
 For simplicity, we have stated Theorem~\ref{thm:parity-tilting} only in the case $\bk$ is a field. But the Satake equivalence also holds when $\bk$ is the ring of integers in a finite extension of $\Ql$, and the notion of tilting objects also makes sense for split reductive group schemes over such rings, see~\cite[\S\S B.9--B.10]{jantzen}. Therefore we can consider the tilting objects in $\Perv_{\GO}(\Gr,\bk)$. On the other hand, the notion of parity objects also makes sense in $\Db_{(\GO)}(\Gr,\bk)$, and their classification is similar in this setting; see~\cite{jmw}. We claim that Theorem~\ref{thm:parity-tilting} also holds for this choice of coefficients.
 
In fact, if $\bk_0$ is the residue field of $\bk$, then it follows from~\cite[Lemma~B.9 \& Lemma~B.10]{jantzen} and the compatibility of the Satake equivalence with extension of scalars that an object $\cF$ in $\Perv_{\GO}(\Gr,\bk)$ is tilting if and only if $\bk_0 \lotimes_\bk \cF$ belongs to $\Perv_{\GO}(\Gr,\bk_0)$ and is tilting therein.
Now if $\cE$ is a parity object in $\Db_{(\GO)}(\Gr,\bk)$, then we have
 \begin{equation}
 \label{eqn:pH0-parity}
  \bk_0 \lotimes_\bk \pH^0(\cE) \cong \pH^0(\bk_0 \lotimes_\bk \cE).
 \end{equation}
 Indeed, assume that $\cE$ is even, and supported on a connected component of $\Gr$ containing $\GO$-orbits of even dimension. (The other cases are similar.) By~\cite[Theorem~1.6 and its proof]{jmw2}, the complex
 \[
  \bk_0 \lotimes_\bk \mathsf{R}^G_T(\cE) \cong \mathsf{R}^G_T(\bk_0 \lotimes_\bk \cE)
 \]
 is an even complex on the affine Grassmannian $\Gr_T$; therefore so is the complex $\mathsf{R}^G_T(\cE)$ by~\cite[Proposition~2.37]{jmw}. In view of~\cite[Lemma~1.10.7]{br}, this shows that $\pH^n(\cE)=0$ and $\pH^n(\bk_0 \lotimes_\bk \cE)=0$ unless $n$ is even. Then~\eqref{eqn:pH0-parity} is an easy consequence of this observation.
 
 From~\eqref{eqn:pH0-parity} and the comments above we obtain the desired extension of Theorem~\ref{thm:parity-tilting}.
\end{rmk}

\subsection{Interpretation in terms of Donkin's tensor product theorem}

In this subsection we assume that $\ell=\mathrm{char}(\bk)$ is good for $G$. 
Recall the triangulated category $\Db_{\IW}(\Fl,\bk)$ introduced in~\S\ref{ss:JMW-conj}.
The $I^+_\uu$-orbits in $\Fl$ are parametrized in a natural way by $W$, and those which support an $(I^+_\uu, \chi_{I^+}^*(\cL^\bk_\psi))$-equivariant local system are the ones corresponding to the elements $w \in W$ which are of minimal length in $\Wf w$. In this case, we denote by $\cE^\IW_w$ the corresponding indecomposable parity object.

As observed in~\S\ref{ss:JMW-conj} (see in particular Remark~\ref{rmk:parities}), under our present assumption, for any $\lambda \in \bX_+^\vee$ the object $\Phi(\cE_\lambda)$ is indecomposable and parity.
Therefore its pullback to $\Fl$ is also parity (by Lemma~\ref{lem:pi-parity}) and indecomposable (by~\cite[Lemma~A.5]{acr}).
We deduce that
\begin{equation}
\label{eqn:Phi-parity}
\pi^* \Phi(\cE_\lambda)[\dim \mathscr{B}] \cong \cE_{t_{\lambda + \varsigma}}^\IW.
\end{equation}
Using the functor $\mathscr{Z}$ considered in~\S\ref{ss:geometric-Weyl}, this formula can also be interpreted as follows.

\begin{prop}
\label{prop:isom-Donkin}
For any $\lambda \in \bX^\vee_+$, we have
\[
\cE^\IW_{t_{\varsigma}} \star^{I^-} \mathscr{Z}(\cE_\lambda) \cong \cE^\IW_{t_{\lambda + \varsigma}}.
\]
\end{prop}

\begin{proof}
 As in the proof of Proposition~\ref{prop:Lusztig-formula-2}, the claim follows from~\eqref{eqn:Phi-parity} using Lem\-ma~\ref{lem:convolution-induction} and Lemma~\ref{lem:induction-Z}.
\end{proof}

Let $\overline{\bk}$ be an algebraic closure of $\bk$, and assume that $\ell$ is strictly bigger than the Coxeter number of $G$. Then
the formula of Proposition~\ref{prop:isom-Donkin} is related to Donkin's tensor product theorem for tilting modules of the Langlands dual $\overline{\bk}$-group $G^\vee_{\overline{\bk}}$ as follows. In~\cite{tilting, prinblock, mkdkm} the authors construct a ``degrading functor''
\[
\eta : \mathsf{Parity}_{\IW}(\Fl^\circ, \overline{\bk}) \to \Tilt_{\mathrm{prin}}(G_{\overline{\bk}}^\vee),
\]
where $\Fl^\circ$ is the connected component of the base point in $\Fl$, $\mathsf{Parity}_{\IW}(\Fl^\circ, \overline{\bk})$ is the category of $(I^+_\uu, \chi_{I^+}^*(\cL^\bk_\psi))$-equivariant parity complexes on $\Fl^\circ$, and $\Tilt_{\mathrm{prin}}(G^\vee_{\overline{\bk}})$ denotes the category of tilting objects in the (non-extended) principal block of the category of finite-dimensional $G^\vee_{\overline{\bk}}$-modules. We expect that Donkin's tensor product theorem (see~\cite[\S E.9]{jantzen}) can be explained geometrically by an isomorphism of complexes involving the functor $\mathscr{Z}$ (see also~\cite[\S 9.3]{surveysmf} for more details). In fact, from this point of view Proposition~\ref{prop:isom-Donkin} is the geometric statement that underlies the isomorphism
\begin{equation}
\label{eqn:isom-Donkin}
 \mathsf{T}(\ell \varsigma) \otimes \mathsf{T}(\lambda)^{(1)} \cong \mathsf{T}(\ell \varsigma + \ell \lambda),
\end{equation}
where $\mathsf{T}(\nu)$ is the indecomposable tilting $G_{\overline{\bk}}^\vee$-module of highest weight $\nu$.

\begin{rmk}
 In general, Donkin's tensor product formula is known at present only when the characteristic of $\overline{\bk}$ is at least $2h-2$, where $h$ is the Coxeter number. However, this restriction is not necessary for the special case~\eqref{eqn:isom-Donkin}. Indeed, as explained in~\cite[Lemma~E.9]{jantzen}, the crucial ingredient to prove~\eqref{eqn:isom-Donkin} is the statement that $\mathsf{T}(\ell \varsigma)$ is indecomposable as a module for the Frobenius kernel $(G^\vee_{\overline{\bk}})_1$ of $G^\vee_{\overline{\bk}}$. However, by~\cite[Proposition~E.11]{jantzen} we have $\mathsf{T}(\ell \varsigma) \cong T_{(\ell-1) \varsigma}^{\ell \varsigma}(\mathsf{T}((\ell-1) \varsigma))$. Now $\mathsf{T}((\ell-1) \varsigma)$ is the Steinberg module $\mathsf{L}((\ell-1)\varsigma)$, and~\cite[\S 11.10]{jantzen} implies that its image under $T_{(\ell-1) \varsigma}^{\ell \varsigma}$
 is indeed indecomposable as a $(G^\vee_{\overline{\bk}})_1$-module.
\end{rmk}




\end{document}